\documentclass[10pt,reqno]{amsart}
\usepackage{pdfsync}
\usepackage{amssymb,amscd,color}

\makeatletter
    
    \@addtoreset{equation}{section}
  \makeatother
                          
\newtheorem{theorem}{Theorem}[section]
\newtheorem{proposition}[theorem]{Proposition}
\newtheorem{corollary}[theorem]{Corollary}
\newtheorem{lemma}[theorem]{Lemma}

\theoremstyle{example}
\newtheorem{example}[theorem]{Example}
\theoremstyle{remark}
\newtheorem{remark}[theorem]{Remark}
\newtheorem{remarks}[theorem]{Remarks}

\theoremstyle{definition}
\newtheorem{definition}{Definition}[section]

\begin{document}
\title[Orbital free pressure and Its Legendre transform]{Orbital free pressure and its Legendre transform}
\author[F.~Hiai]
{Fumio Hiai}
\address{(F.H.) Tohoku University (Emeritus),
Hakusan 3-8-16-303, Abiko 270-1154, Japan}
\email{hiai.fumio@gmail.com}
\author[Y.~Ueda]
{Yoshimichi Ueda$\,^{1}$}
\address{(Y.U.) Graduate School of Mathematics, 
Kyushu University, 
Fukuoka, 819-0395, Japan}
\email{ueda@math.kyushu-u.ac.jp}
\thanks{$^{1}$Supported in part by Grant-in-Aid for Scientific Research (C) 24540214.}
\maketitle

\begin{abstract} Orbital counterparts of the free pressure and its Legendre transform (or $\eta$-entropy) are introduced and studied in comparison with other entropy quantities in free probability theory and in relation to random multi-matrix models. 
\end{abstract}

\allowdisplaybreaks{

\section{Introduction} 

Voiculescu \cite{Voiculescu:AdvMath99} introduced the \emph{free mutual information} $i^*$ for tuples of unital $*$-subalgebras of a tracial $W^*$-probability space, based on the so-called liberation processes and an appropriate theory of derivations similarly to the microstate-free approach $\chi^*$ to free entropy. Hence it is natural to regard the free mutual information $i^*$ as the microstate-free definition of free mutual information. 
A decade later in \cite{HiaiMiyamotoUeda:IJM09}, we and Miyamoto introduced the \emph{orbital free entropy} $\chi_\mathrm{orb}$ for tuples $(\mathbf{X}_1,\dots,\mathbf{X}_n)$ of self-adjoint random multi-variables $\mathbf{X}_i$ in a tracial $W^*$-probability space, based on an appropriate notion of microstates called \emph{orbital microstates}. Here, a self-adjoint random multi-variable means a tuple of (finitely many) self-adjoint random variables. Note that our first work \cite{HiaiMiyamotoUeda:IJM09} could deal with only the situation that each $\mathbf{X}_i$ generates a hyperfinite von Neumann subalgebra, but the later one \cite{Ueda:IUMJ1x}, motivated by the important contribution \cite{BianeDabrowski:AdvMath13}, extended the orbital free entropy $\chi_\mathrm{orb}$ itself with its properties and its dimension counterpart $\delta_{\mathrm{orb},0}$ to the general situation. Furthermore, several alternative approaches to orbital free entropy, $\tilde{\chi}_\mathrm{orb}$ etc., were proposed by Biane and Dabrowski \cite{BianeDabrowski:AdvMath13} as byproducts of their idea of random microstates in order to provide paths to solving at least one of the major problems in the direction. Let us briefly explain those major problems in the next paragraph to make the current status clear for the reader's convenience. 

The free mutual information $i^*$ and the minus orbital free entropy $-\chi_\mathrm{orb}$ have many natural properties in common, and hence $-\chi_\mathrm{orb}$ seems a strong candidate of microstate counterpart of $i^*$. However, only a few results are known so far about relations among them and two approaches $\chi$, $\chi^*$ of free entropy. In fact, the most desired identity $\chi(\mathbf{X}_1,\dots,\mathbf{X}_n) = \chi_\mathrm{orb}(\mathbf{X}_1,\dots,\mathbf{X}_n) + \sum_{i=1}^n \chi(\mathbf{X}_i)$ is known to hold only when every $\mathbf{X}_i$ is a singleton (see \cite[Theorem 2.6]{HiaiMiyamotoUeda:IJM09}), and moreover no analogous property for $\chi^*$ and $i^*$ is known. This issue is currently one of two major problems in the study of orbital free entropy $\chi_\mathrm{orb}$, and the other is the unification question between $i^*$ and $-\chi_\mathrm{orb}$ (see \cite[\S1]{IzumiUeda:Preprint13} for more on this latter problem). Besides these fundamental problems we do not yet know whether or not all the existing microstate approaches or definitions in \cite{BianeDabrowski:AdvMath13,Ueda:IUMJ1x} coincide without any assumption; hence we (and probably the authors of \cite{BianeDabrowski:AdvMath13} too) think that the `final form' of definition of orbital free entropy $\chi_\mathrm{orb}$ is not yet established. Moreover, it is very difficult to calculate the orbital free entropy $\chi_\mathrm{orb}$ (as well as the free mutual information $i^*$) for concrete, non-trivial examples. The latter of these two issues initially motivated us to do the present work.  

The main purpose of the present paper is to introduce and develop a possible Legendre transform approach to the orbital free entropy $\chi_\mathrm{orb}$ in order to deepen our understanding of $\chi_\mathrm{orb}$ (as well as its relation with the microstate approach $\chi$ to free entropy). The resulting quantity turns out not to agree with the orbital free entropy $\chi_\mathrm{orb}$ in general due to its character of concavity, but shares all the natural properties with $\chi_\mathrm{orb}$. We would like to emphasize that the proposed approach here enables one to investigate the orbital free entropy $\chi_\mathrm{orb}$ by analogy to statistical mechanics with a flavor similar to quantum spin systems \cite[Chapter 6]{BratteliRobinson:Book2}. 

Let us clarify our strategy and our results more precisely. The strategy here is to follow an idea from statistical mechanics, and this is analogous to the previous work \cite{Hiai:CMP05} of the first-named author. Namely, we first introduce a kind of pressure function, called the \emph{orbital free pressure} $\pi_{\mathrm{orb},R}$, in the sprit of defining the orbital free entropy $\chi_\mathrm{orb}$, and then consider its (minus) Legendre transform. Hence the resulting quantity fits into convex analysis and is primarily defined for tracial states on a certain universal $C^*$-algebra depending on a fixed cut-off constant. However, we see (Theorem \ref{T3.4}) that it does not depend on the possible choice of a cut-off constant while this independence is not known for the previous \emph{$\eta$-entropy} $\eta_R$ introduced in \cite{Hiai:CMP05}, and it turns out to become a new variant of the orbital free entropy. We will call it the \emph{orbital $\eta$-entropy} $\eta_\mathrm{orb}$, defined for tuples $(\mathbf{X}_1,\dots,\mathbf{X}_n)$ of self-adjoint random multi-variables as above. In fact, we see (Theorem \ref{T3.5}) that the orbital $\eta$-entropy $\eta_\mathrm{orb}$ shares all the natural properties with the orbital free entropy $\chi_\mathrm{orb}$. The inequality $\chi_\mathrm{orb}\le\eta_\mathrm{orb}$ holds in general, though $\eta_\mathrm{orb}$ is not necessarily equal to $\chi_\mathrm{orb}$ (see Remark \ref{R3.6}) by the concavity as mentioned before. A systematic study of $\pi_{\mathrm{orb},R}$ and $\eta_\mathrm{orb}$ including the above-mentioned results is carried out in \S2 and \S3, and the purpose of these sections and the later \S5 is to establish the foundation of the Legendre transform approach here. In the next \S4 we introduce and study the notion of `orbital equilibrium' tracial states. In particular, in \S\S4.2 we give a sufficient condition for a given tracial state to be orbital equilibrium, which will be used in \S7 where we will point out (in Example 7.1) an explicit relationship between a computation in the proof of \cite[Theorem 8.1]{CollinsGuionnetSegala:AdvMath09} and the orbital free entropy $\chi_\mathrm{orb}$. Indeed, the proof of \cite[Theorem 8.1]{CollinsGuionnetSegala:AdvMath09} contains a `prototype' of the orbital free pressure and it is
one of our motivations of the present work. Moreover, in \S\S4.3 we find several connections of the orbital $\eta$-entropy $\eta_\mathrm{orb}$ to the microstate approach $\chi$ of free entropy and the $\eta$-entropy $\eta_R$ based on the notion of (orbital) equilibrium. In \S5 we prove (Theorem \ref{T5.1}) that the free independence of $\mathbf{X}_1,\dots,\mathbf{X}_n$ (together with the f.d.a.\ assumption) is equivalent to $\eta_\mathrm{orb}(\mathbf{X}_1,\dots,\mathbf{X}_n) = 0$, where the key ingredient is a transportation cost inequality similarly to the case of $\chi_\mathrm{orb}$ in \cite{HiaiMiyamotoUeda:IJM09}. In \S6 we give a representation of the orbital free entropy $\chi_\mathrm{orb}$ as the Legendre transform of a certain modification of the pressure function introduced in \S2. This is again analogous to the representation in \cite[\S6]{Hiai:CMP05}. Note that such a representation is important as explained in \cite[\S7]{Guionnet:LNM09}. In the final \S7 we discuss random matrix models studied in \cite{CollinsGuionnetSegala:AdvMath09}. We see that random matrix models in \cite{CollinsGuionnetSegala:AdvMath09} under some conditions indeed provide orbital equilibrium tracial states in a suitable manner. 

Many existing quantities mentioned above will be used throughout the paper, but we would not like to repeat their definitions here. Instead we simply refer the reader to suitable references as follows: \cite{Voiculescu:Survey} for two approaches $\chi, \chi^*$ to free entropy,  \cite{Hiai:CMP05} for the free pressure $\pi_R$ and the $\eta$-entropy $\eta_R$, \cite{HiaiMiyamotoUeda:IJM09,Ueda:IUMJ1x} for the orbital free entropy $\chi_\mathrm{orb}$, and finally \cite{Voiculescu:AdvMath99} for the liberation theory. We will use the following notations: 
For a given unital $C^*$-algebra $\mathcal{A}$ we denote by $\mathcal{A}^{sa}$ the set of all self-adjoint elements in $\mathcal{A}$ and by $TS(\mathcal{A})$ the tracial states on $\mathcal{A}$.    
For $N\in\mathbb{N}$, $M_N(\mathbb{C})$ is the $N\times N$ complex matrix algebra, $M_N(\mathbb{C})^{sa}$ is the set of $N\times N$ self-adjoint matrices, and $\Lambda_N$ is the Lebesgue measure on $M_N(\mathbb{C})^{sa} \cong \mathbb{R}^{N^2}$. We write
$M_N(\mathbb{C})_R^{sa}$ for the set of $A\in M_N(\mathbb{C})^{sa}$ with the operator
norm $\|A\|_\infty\le R$. The normalized trace on $M_N(\mathbb{C})$ is denoted by
$\mathrm{tr}_N$. Furthermore, $\mathrm{U}(N)$ and $\mathrm{SU}(N)$ are the unitary group and the special unitary group of order $N$ with the
Haar probability measures $\gamma_{\mathrm{U}(N)}$ and $\gamma_{\mathrm{SU}(N)}$, respectively.

\section{Orbital free pressure}

Let $\mathbf{x}_i = (x_{ij})_{j=1}^{r(i)}$, $1 \leq i \leq n$, be non-commutative
multi-indeterminates, and let $R > 0$ be given. We define $\mathcal{C}_R(\mathbf{x}_i)$ to
be the universal $C^*$-free product of $r(i)$  copies of $C[-R,R]$, the $C^*$-algebra of
continuous complex functions on $[-R,R]$, with the identification $x_{ij}(t) := t$ in
the $j$th copy of $C[-R,R]$. (Recall that the universal $C^*$-algebra generated by a single element $x=x^*$ with $\|x\|=R$ is isomorphic to $C[-R,R]$ with the identification
$x(t)=t$.) Write $\mathbf{x} := \mathbf{x}_1\sqcup\cdots\sqcup\mathbf{x}_n$ for simplicity, and
define $\mathcal{C}_R(\mathbf{x})$ to be the universal $C^*$-free product of the
$\mathcal{C}_R(\mathbf{x}_i)$, $1 \leq i \leq n$, i.e.,
\begin{equation}\label{F1}
\mathcal{C}_R(\mathbf{x})
=\bigstar_{i=1}^n\mathcal{C}_R(\mathbf{x}_i)
=\bigstar_{i=1}^n\bigl(C[-R,R]^{\star r(i)}\bigr).
\end{equation}
Denote by $\Vert-\Vert_R$ the $C^*$-norm on those $C^*$-algebras throughout.

For a given tracial state
$\tau \in TS(\mathcal{C}_R(\mathbf{x}_i))$ and $N,m\in\mathbb{N}$ and $\delta>0$ we define
$\Gamma_R(\tau;N,m,\delta)$ to be the set $\Gamma_R(\mathbf{x}_i;N,m,\delta)$  of microstates with regarding $\mathbf{x}_i$ as a tuple of random variables in the tracial probability space
$(\mathcal{C}_R(\mathbf{x}_i),\tau)$, i.e., the set of
$(A_j)_{j=1}^{r(i)}\in (M_N(\mathbb{C})_R^{sa})^{r(i)}$ such that
$$
|\mathrm{tr}_N(A_{j_1}\cdots A_{j_l})-\tau(x_{ij_1}\cdots x_{ij_l})|<\delta
$$
holds whenever $1 \leq j_k \leq r(i)$, $1 \leq k \leq l$ and $l\le m$. As noted in \cite[p.232--233]{Hiai:CMP05},
for each $\mathbf{a} = (\mathbf{a}_i)_{i=1}^n$ with tuples $\mathbf{a}_i = (a_{ij})_{j=1}^{r(i)}$ of self-adjoint elements in a unital $C^*$-algebra $\mathcal{A}$ and $\|a_{ij}\|\le R$, there is a canonical
$*$-homomorphism $f \in \mathcal{C}_R(\mathbf{x}) \mapsto f(\mathbf{a}) \in \mathcal{A}$ uniquely
determined by the correspondence $x_{ij} \mapsto a_{ij}$. Similarly, we have
$*$-homomorphisms $g\in\mathcal{C}_R(\mathbf{x}_i) \mapsto g(\mathbf{a}_i)\in\mathcal{A}$
for $1\le i\le n$. In particular, we have $*$-homomorphisms
$h\in \mathcal{C}_R(\mathbf{x}_i) \mapsto h(\mathbf{A}_i)\in M_N(\mathbb{C})$ for
$\mathbf{A}=(\mathbf{A}_i)_{i=1}^n$ with
$\mathbf{A}_i=(A_{ij})_{i=1}^{r(i)}\in(M_N(\mathbb{C})_R^{sa})^{r(i)}$.

\begin{definition}\label{D2.1}
Let $\tau_i \in TS(\mathcal{C}_R(\mathbf{x}_i))$, $1 \leq i \leq n$, be given, and
$h \in \mathcal{C}_R(\mathbf{x})^{sa}$ be also given. For each $N,m\in\mathbb{N}$ and $\delta>0$ we define
\begin{align} 
&\pi_{\mathrm{orb},R}(h:(\tau_i)_{i=1}^n\,;N,m,\delta) \nonumber\\
&:= 
\sup_{(\mathbf{A}_i)_{i=1}^n \in \prod_{i=1}^n\Gamma_R(\tau_i\,;N,m,\delta)} 
\log \int_{\mathrm{U}(N)^n} d\gamma_{\mathrm{U}(N)}^{\otimes n}(V_i)
\exp\bigl(-N^2\mathrm{tr}_N(h((V_i\mathbf{A}_i V_i^*)_{i=1}^n))\bigr), \label{F2}
\end{align}
where the supremum should be read $-\infty$ when $\Gamma_R(\tau_i\,;N,m,\delta) =\emptyset$ for some $1 \leq i \leq n$. Moreover, we define
\begin{align*}
&\pi_{\mathrm{orb},R}(h:(\tau_i)_{i=1}^n\,;m,\delta) 
:= 
\limsup_{N\to\infty} \frac{1}{N^2}\,\pi_{\mathrm{orb},R}
(h:(\tau_i)_{i=1}^n\,;N,m,\delta), \\
&\pi_{\mathrm{orb},R}(h:(\tau_i)_{i=1}^n) 
:= 
\lim_{m\to\infty\atop\delta\searrow0} \pi_{\mathrm{orb},R}(h:(\tau_i)\,;m,\delta)
= \inf_{m \in \mathbb{N}\atop\delta>0} \pi_{\mathrm{orb},R}(h:(\tau_i)\,;m,\delta).
\end{align*}
We call $\pi_{\mathrm{orb},R}(h:(\tau_i)_{i=1}^n)$ the \emph{orbital free pressure of $h$} 
relative to $(\tau_i)_{i=1}^n$. 
\end{definition}

For $\tau_i\in TS(\mathcal{C}_R(\mathbf{x}_i))$ we say that $\tau_i$ has
\emph{finite-dimensional approximants} (f.d.a.\ for short) \cite[Definition 3.1]{Voiculescu:IMRN98} if, for every $m\in\mathbb{N}$
and $\delta>0$, $\Gamma_R(\tau_i;N,m,\delta)\ne\emptyset$ for some $N\in\mathbb{N}$. The
next lemma is well-known and the proof is left to the reader.

\begin{lemma}\label{L2.1}
The following conditions for $\tau_i\in\mathcal{C}_R(\mathbf{x}_i)$ are equivalent{\rm:}
\begin{itemize}
\item[(i)] $\tau_i$ has f.d.a.{\rm;}
\item[(ii)] for every $m\in\mathbb{N}$ and $\delta>0$,
$\Gamma_R(\tau_i;N,m,\delta)\ne\emptyset$ for all sufficiently large $N\in\mathbb{N}${\rm;}
\item[(iii)] there exists a sequence of microstates
$\Xi(N)=(\xi_j(N))_{j=1}^{r(i)}\in(M_N(\mathbb{C})_R^{sa})^{r(i)}$, $N\in\mathbb{N}$, such that $\lim_{N\to\infty}\mathrm{tr}_N(p(\Xi(N)))=\tau_i(p)$ for every
non-commutative polynomial $p$ in $\mathbf{x}_i$.
\end{itemize}
\end{lemma}
%\begin{proof}
%It is trivial that (iii) $\Rightarrow$ (ii) $\Rightarrow$ (i). Assume (i) and let $m \in \mathbb{N},\delta > 0$ be arbitrary. Then there exist $N_0\in\mathbb{N}$ and $\mathbf{A} = (A_j)_{j=1}^{r(i)} \in (M_{N_0}(\mathbb{C})_R^{sa})^{r(i)}$ such that $|\mathrm{tr}_{N_0}(p(\mathbf{A})) - \tau(p(\mathbf{x}_i))| < \delta$ holds for every monomial $p$ in $\mathbf{x}_i$ of degree not greater than $m$. Let $\delta'$ be the maximum of $|\mathrm{tr}_{N_0}(p(\mathbf{A})) - \tau(p)|$ over all those monomials $p$, which must be strictly smaller than $\delta$. For each $N = nN_0 + l$ with $0 \leq l \leq N_0 -1$, define $A_j^{(N)} := (\bigoplus^n A_j)\oplus O_{l\times l} \in M_N(\mathbb{C})^{sa}_R$, $ 1 \leq j \leq r(i)$, and set $\mathbf{A}^{(N)} := (A_j^{(N)})_{j=1}^{r(i)}$. Then we easily see that $|\mathrm{tr}_N(p(\mathbf{A}^{(N)}) - \tau(p)| \leq \delta' + R^m/n$. Therefore, $\mathbf{A}^{(N)}$ falls in $\Gamma_R(\tau_i\,;N,m,\delta)$ as long as $N \geq n_0 N_0$ with $n_0 > R^m/(\delta-\delta')$. Hence (ii) follows.

%Assume (ii); then for each $m\in\mathbb{N}$ there is an $N_m\in\mathbb{N}$ so that $\Gamma_R(\tau_i;N,m,1/m)\ne\emptyset$ for all $N\ge N_m$.  We may assume that $N_1<N_2<\dots$, and choose $\Xi(N)\in\Gamma_R(\tau_i;N,m,1/m)$ for $N_m\le N<N_{m+1}$. This gives assertion (iii).
%\end{proof}

Basic properties of the orbital free pressure are in the following:
 
\begin{proposition}\label{P2.2} The orbital free pressure
$\pi_{\mathrm{orb},R}(h:(\tau_i)_{i=1}^n)$ enjoys the following properties{\rm:}
\begin{itemize}
\item[(1)] If one of the $\tau_i$ does not have f.d.a., then
$\pi_{\mathrm{orb},R}(h:(\tau_i)_{i=1}^n)=-\infty$ for all
$h\in\mathcal{C}_R(\mathbf{x})^{sa}$.
\item[(2)] Assume that all $\tau_i$ have f.d.a. Then, for every
$h_1, h_2 \in \mathcal{C}_R(\mathbf{x})^{sa}$ one has
$$
|\pi_{\mathrm{orb},R}(h_1:(\tau_i)_{i=1}^n) - \pi_{\mathrm{orb},R}(h_2:(\tau_i)_{i=1}^n)| \leq \Vert h_1 - h_2\Vert_R.
$$
In particular, $\pi_{\mathrm{orb},R}(h:(\tau_i)_{i=1}^n)\in[-\|h\|_R,\|h\|_R]$ for every
$h \in \mathcal{C}_R(\mathbf{x})^{sa}$.
\item[(3)] If $h_1, h_2 \in \mathcal{C}_R(\mathbf{x})^{sa}$ and $h_1 \leq h_2$, then $\pi_{\mathrm{orb},R}(h_1:(\tau_i)_{i=1}^n) \geq \pi_{\mathrm{orb},R}(h_2:(\tau_i)_{i=1}^n)$.
\item[(4)] $h \in \mathcal{C}_R(\mathbf{x})^{sa} \mapsto \pi_{\mathrm{orb},R}(h:(\tau_i)_{i=1}^n)$ is convex.  
\item[(5)] For $1\le n'<n$ and for every $h_1 \in \mathcal{C}_R(\bigsqcup_{i=1}^{n'}\mathbf{x}_i)^{sa}$ and
$h_2 \in \mathcal{C}_R(\bigsqcup_{i=n'+1}^n\mathbf{x}_i)^{sa}$ one has
$$
\pi_{\mathrm{orb},R}(h_1+h_2:(\tau_i)_{i=1}^n) \leq 
\pi_{\mathrm{orb},R}(h_1:(\tau_i)_{i=1}^{n'})
+ \pi_{\mathrm{orb},R}(h_2:(\tau_i)_{i=n'+1}^n). 
$$
\end{itemize}
\end{proposition}
\begin{proof}
(1) is obvious by definition and Lemma \ref{L2.1}.

(2) For every $(\mathbf{A}_i)_{i=1}^n \in \prod_{i=1}^n\Gamma_R(\tau_i;N,m,\delta)$ and
$(V_i)_{i=1}^n \in \mathrm{U}(N)^n$, the obvious inequality
$$
\big|\mathrm{tr}_N(h_1((V_i\mathbf{A}_i V_i^*)_{i=1}^n))
- \mathrm{tr}_N(h_2((V_i\mathbf{A}_i V_i^*)_{i=1}^n))| \leq \Vert h_1 - h_2\Vert_R
$$ 
implies
\begin{align*}
&\exp\bigl(-N^2\mathrm{tr}_N\bigl(h_1((V_i\mathbf{A}_iV_i^*)_{i=1}^n)\bigr)\bigr)
e^{-\|h_1-h_2\|_R} \\
&\quad\le\exp\bigl(-N^2\mathrm{tr}_N\bigl(h_2((V_i\mathbf{A}_iV_i^*)_{i=1}^n)\bigr)\bigr)
\le\exp\bigl(-N^2\mathrm{tr}_N\bigl(h_1((V_i\mathbf{A}_iV_i^*)_{i=1}^n)\bigr)\bigr)
e^{\|h_1-h_2\|_R},
\end{align*}
which immediately gives the first assertion. The latter assertion also follows since
$\pi_{\mathrm{orb},R}(0:(\tau_i)_{i=1}^n)=0$.

(3) is easy and the proof is left to the reader.

(4) We may assume, thanks to (1) above, that all $\tau_i$ have f.d.a. Let
$0 < \alpha < 1$ and $h_1, h_2 \in \mathcal{C}_R(\mathbf{x})^{sa}$ be arbitrarily fixed.
For any $\mathbf{A}_i \in \Gamma(\tau_i;N,m,\delta)$, $1 \leq i \leq n$, one has
\begin{align*} 
&\int_{\mathrm{U}(N)^n} d\gamma_{\mathrm{U}(N)}^{\otimes n}(V_i)
\exp\big(-N^2\mathrm{tr}_N\big((\alpha h_1 + (1-\alpha)h_2)((V_i\mathbf{A}_i V_i^*)_{i=1}^n)\big)\big) \\
&\quad= 
\int_{\mathrm{U}(N)^n} d\gamma_{\mathrm{U}(N)}^{\otimes n}(V_i) \exp\big(-N^2\mathrm{tr}_N\big(h_1((V_i\mathbf{A}_i V_i^*)_{i=1}^n)\big)\big)^\alpha \\
&\phantom{aaaaaaaaaaaaaaaaaaa}\times\exp\big(-N^2\mathrm{tr}_N\big(h_2((V_i\mathbf{A}_i V_i^*)_{i=1}^n)\big)\big)^{1-\alpha} \\
&\quad\leq 
\Big[\int_{\mathrm{U}(N)^n} d\gamma_{\mathrm{U}(N)}^{\otimes n}(V_i)
\exp\big(-N^2\mathrm{tr}_N\big(h_1((V_i\mathbf{A}_i V_i^*)_{i=1}^n)\big)\big) \Big]^\alpha \\
&\phantom{aaaaaaaa}
\times \Big[\int_{\mathrm{U}(N)^n} d\gamma_{\mathrm{U}(N)}^{\otimes n} \exp\big(-N^2\mathrm{tr}_N\big(h_2((V_i\mathbf{A}_i V_i^*)_{i=1}^n)\big)\big) \Big]^{1-\alpha}
\end{align*}
by the H\"{o}lder inequality, and hence
\begin{align*}
&\log\int_{\mathrm{U}(N)^n} d\gamma_{\mathrm{U}(N)}^{\otimes n}(V_i) \exp\big(-N^2\mathrm{tr}_N\big((\alpha h_1 + (1-\alpha)h_2)((V_i\mathbf{A}_i V_i^*)_{i=1}^n)\big)\big) \\
&\qquad\qquad\qquad
\leq 
\alpha \pi_{\mathrm{orb},R}(h_1:(\tau_i)_{i=1}^n;N,m,\delta) + (1-\alpha)\pi_{\mathrm{orb},R}(h_2:(\tau_i)_{i=1}^n;N,m,\delta). 
\end{align*} 
Then the desired assertion is immediate.  

(5) We may and do also assume that all the $\tau_i$ have f.d.a.
For any $\mathbf{A}_i \in \Gamma(\tau_i;N,m,\delta)$, $1 \leq i \leq n$, one has 
\begin{align*} 
&\log\int_{\mathrm{U}(N)^n}d\gamma_{\mathrm{U}(N)}^{\otimes n}(V_i) \exp\bigl(-N^2 \bigl(\mathrm{tr}_N(h_1 ((V_i \mathbf{A}_i V_i^*)_{i=1}^{n'}) + \mathrm{tr}_N(h_2 ((V_i \mathbf{A}_i V_i^*)_{i=n'+1}^n)\bigr)\bigr) \\
&\quad= 
\log\int_{\mathrm{U}(N)^{n'}}d\gamma_{\mathrm{U}(N)}^{\otimes n'}(V_i) \exp\bigl(-N^2 \mathrm{tr}_N(h_1 ((V_i \mathbf{A}_i V_i^*)_{i=1}^{n'}))\bigr) \\
&\qquad\qquad+ \log\int_{\mathrm{U}(N)^{n-n'}}d\gamma_{\mathrm{U}(N)}^{\otimes(n-n')}(V_i) \exp\bigl(-N^2\mathrm{tr}_N(h_2 ((V_i \mathbf{A}_i V_i^*)_{i=n'+1}^n))\bigr) \\
&\quad\leq 
\pi_{\mathrm{orb},R}(h_1:(\tau_i)_{i=1}^{n'}\,;N,m,\delta) + \pi_{\mathrm{orb},R}(h_2:(\tau_i)_{i=n'+1}^n\,;N,m,\delta),
\end{align*} 
which immediately implies the desired assertion. 
\end{proof}

One can also consider $\pi_{\mathrm{orb},R}(h:(\tau_i)_{i=1}^n)$ as a function of
$\tau_i\in TS(\mathcal{C}_R(\mathbf{x}_i))$, $1\le i\le n$, with
$h\in\mathcal{C}_R(\mathbf{x})^{sa}$ fixed. The following lemma will be used in the
next section.

\begin{lemma}\label{L2.3}
For $1\le i\le n$ let $\tau_i,\tau_i^{(k)}\in TS(\mathcal{C}_R(\mathbf{x}_i))$,
$k\in\mathbb{N}$, such that $\tau_i^{(k)}\to\tau_i$ in the weak* topology as $k\to\infty$. Then, for
every $h\in\mathcal{C}_R(\mathbf{x})^{sa}$ one has
$$
\pi_{\mathrm{orb},R}(h:(\tau_i)_{i=1}^n)
\ge\limsup_{k\to\infty}\pi_{\mathrm{orb},R}(h:(\tau_i^{(k)})_{i=1}^n).
$$
\end{lemma}

\begin{proof}
Let $m \in \mathbb{N}$ and $\delta>0$ be arbitrarily given. By assumption we can choose
$k_0$ (depending on $m, \delta$) so that, for every $k \geq k_0$ and $1 \leq i \leq n$, 
$$
|\tau_i^{(k)}(x_{ij_1}\cdots x_{ij_l}) - \tau_i(x_{ij_1}\cdots x_{ij_l})| < \delta/2
$$
holds whenever $1 \leq j_k \leq r(i)$, $1 \leq k \leq l$ and $1 \leq l \leq m$. For every $k \geq k_0$
and $1 \leq i \leq n$, if
$(B_j)_{j=1}^{r(i)} \in \Gamma_R(\tau_i^{(k)}\,;N,m,\delta/2)$, then we have
$$ 
\big|\mathrm{tr}_N(B_{j_1}\cdots B_{j_l}) - \tau_i(x_{ij_1}\cdots x_{ij_l})\big| 
\leq \big|\mathrm{tr}_N(B_{j_1}\cdots B_{j_l}) - \tau_i^{(k)}(x_{ij_1}\cdots x_{ij_l})\big|
+ \delta/2 < \delta
$$
for the $j_k$ as above. This implies that
$\Gamma_R(\tau_i^{(k)}\,;N,m,\delta/2) \subseteq \Gamma_R(\tau_i\,;N,m,\delta)$
holds for every $k \geq k_0$, $1 \leq i \leq n$ and $N \in \mathbb{N}$. Consequently,
for every $h\in\mathcal{C}_R(\mathbf{x})^{sa}$,
$$
\pi_{\mathrm{orb},R}(h:(\tau_i^{(k)})_{i=1}^n)
\le \pi_{\mathrm{orb},R}(h:(\tau_i^{(k)})_{i=1}^n\,;m,\delta/2)
\leq \pi_{\mathrm{orb},R}(h:(\tau_i)_{i=1}^n\,;m,\delta)
$$ 
holds whenever $k \geq k_0$. Therefore,
$$
\limsup_{k\to\infty}\pi_{\mathrm{orb},R}(h:(\tau_i^{(k)})_{i=1}^n)
\le\pi_{\mathrm{orb},R}(h:(\tau_i)_{i=1}^n\,;m,\delta),
$$
which gives the desired inequality since $m\in\mathbb{N}$ and $\delta>0$ are arbitrary.
\end{proof}

Let $\tau_i$, $1\le i\le n$, be again fixed as in Definition \ref{D2.1}. We assume that
every $\mathbf{x}_i$ generates a hyperfinite von Neumann algebra in the GNS representation
associated with $\tau_i$. Obviously, this assumption is stronger than that all $\tau_i$
have f.d.a. For each $1 \leq i \leq n$ we can choose a sequence of microstates as in
Lemma \ref{L2.1}\,(3), i.e.,
$\Xi_i(N) \in (M_N(\mathbb{C})_R^{sa})^{r(i)}$, $N\in\mathbb{N}$, such that
$\lim_{N\to\infty} \mathrm{tr}_N(p(\Xi_i(N))) = \tau_i(p)$ for every non-commutative
polynomial $p$ in $\mathbf{x}_i$. In this situation we have the next result from the
viewpoint similar to \cite[\S4]{HiaiMiyamotoUeda:IJM09}.

\begin{proposition}\label{P2.4} With the assumption and the notations above, for every
$h\in\mathcal{C}_R(\mathbf{x})^{sa}$ we have 
\begin{align*}
&\pi_{\mathrm{orb},R}(h:(\tau_i)_{i=1}^n) \\
&\quad= 
\limsup_{N\to\infty}\frac{1}{N^2}\log
\int_{\mathrm{U}(N)^n} d\gamma_{\mathrm{U}(N)}^{\otimes n}(V_i)
\exp\bigl(-N^2 \mathrm{tr}_N(h((V_i \Xi_i(N) V_i^*)_{i=1}^n))\bigr)
\end{align*}
{\rm(}independently of the choice of approximating microstates $\Xi_i(N)$ as above{\rm)}.
\end{proposition}
\begin{proof} For each $m \in \mathbb{N}$ and $\delta > 0$ there is an $N_0 \in \mathbb{N}$ so that $\Xi_i(N) \in \Gamma_R(\tau_i\,;N,m,\delta)$ for all $N \geq N_0$ and $1 \leq i \leq n$. Hence the right-hand side of the desired identity is not greater than $\pi_{\mathrm{orb},R}(h:(\tau_i)_{i=1}^n\,;m,\delta)$ for every $m \in \mathbb{N}$ and $\delta > 0$. Hence the inequality $\geq$ of the desired identity holds true. 

Let $\varepsilon > 0$ be arbitrarily chosen. One can choose a non-commutative polynomial
$p = p^*$ in $\mathbf{x}$ such that $\Vert p - h\Vert_R < \varepsilon/3$. By Jung's theorem
\cite{Jung:MathAnn07} (see \cite[Lemma 1.2]{HiaiMiyamotoUeda:IJM09} and the proof of
\cite[Proposition 2.3]{Ueda:IUMJ1x}) one can choose (by looking at the polynomial $p$)
$m' \in \mathbb{N}$ and $\delta' > 0$ in such a way that, if $(\mathbf{B}_i)_{i=1}^n, (\mathbf{B}'_i)_{i=1}^n$ are in $\prod_{i=1}^n \Gamma_R(\tau_i\,;N,m',\delta')$, then there exists $W_i \in \mathrm{U}(N)$ (depending on $\mathbf{B}_i,\mathbf{B}'_i$),
$1 \leq i \leq n$, such that 
$$
|\mathrm{tr}_N(p((V_i\mathbf{B}_i V_i^*)_{i=1}^n)) - \mathrm{tr}_N(p((V_i W_i\mathbf{B}'_i W_i^* V_i^*)_{i=1}^n))| < \varepsilon/3 
$$
for every $N \in \mathbb{N}$ and every $(V_i)_{i=1}^n \in \mathrm{U}(N)^n$. Let
$N_1 \in \mathbb{N}$ be chosen so that every $\Xi_i(N)$ falls in
$\Gamma_R(\tau_i\,;N,m',\delta')$ as long as $N \geq N_1$. Now, assume that $N \geq N_1$
and $(\mathbf{A}_i)_{i=1}^n \in \prod_{i=1}^n \Gamma_R(\tau_i\,;N,m',\delta')$, and
choose $W_i \in\mathrm{U}(N)$ as above for $\mathbf{A}_i$ and $\Xi_i(N)$, $1\le i\le n$.
Since
\begin{align*}
&|\mathrm{tr}_N(h((V_i \mathbf{A}_i V_i^*)_{i=1}^n))
- \mathrm{tr}_N(h((V_i W_i\Xi_i(N)W_i^* V_i^*)_{i=1}^n))| \\
&\quad\leq 
2\Vert h - p\Vert_R + 
|\mathrm{tr}_N(p((V_i \mathbf{A}_i V_i^*)_{i=1}^n))
- \mathrm{tr}_N(p((V_i W_i\Xi_i(N)W_i^* V_i^*)_{i=1}^n))| < \varepsilon
\end{align*}
for every $V_i\in\mathrm{U}(N)$, it follows that
\begin{align*}
&\log\int_{\mathrm{U}(N)^n} d\gamma_{\mathrm{U}(N)}^{\otimes n}(V_i)
\exp\bigl(-N^2\mathrm{tr}_N(h((V_i \mathbf{A}_i V_i^*)_{i=1}^n))\bigr) \\
&\quad< 
\log\int_{\mathrm{U}(N)^n} d\gamma_{\mathrm{U}(N)}^{\otimes n}(V_i)
\exp\bigl(-N^2\mathrm{tr}_N(h((V_i W_i \Xi_i(N) W_i^* V_i^*)_{i=1}^n))\bigr)
+ N^2\varepsilon \\
&\quad= 
\log\int_{\mathrm{U}(N)^n} d\gamma_{\mathrm{U}(N)}^{\otimes n}(V_i)
\exp\bigl(-N^2\mathrm{tr}_N(h((V_i \Xi_i(N) V_i^*)_{i=1}^n))\bigr)
+ N^2\varepsilon
\end{align*}
thanks to the invariance of $\gamma_{\mathrm{U}(N)}$. Therefore, for every $N \geq N_1$,
\begin{align*} 
&\pi_{\mathrm{orb},R}(h:(\tau_i)_{i=1}^n)
\leq  
\pi_{\mathrm{orb},R}(h:(\tau_i)_{i=1}^n\,;m',\delta') \\
&\quad\leq 
\limsup_{N\to\infty}\frac{1}{N^2}\log \int_{\mathrm{U}(N)^n}
d\gamma_{\mathrm{U}(N)}^{\otimes n}(V_i)
\exp\bigl(-N^2\mathrm{tr}_N(h((V_i \Xi_i(N) V_i^*)_{i=1}^n))\bigr) + \varepsilon, 
\end{align*} 
which implies the inequality $\leq$ of the desired identity since $\varepsilon>0$ is
arbitrary.  
\end{proof}

\begin{remark}\label{R2.5}
In the hyperfiniteness situation (in particular, in the case where every $\mathbf{x}_i$
is a singleton), the identity of the above proposition may serve as an alternative
definition of $\pi_{\mathrm{orb},R}(h:(\tau_i)_{i=1}^n)$. By suitably modifying the above proof we also see, in this situation, that the definition becomes equivalent when supremum
in \eqref{F2} of Definition \ref{D2.1} is replaced with infimum, that is,
\begin{align*}
&\underline{\pi}_{\mathrm{orb},R}(h:(\tau_i)_{i=1}^n;N,m,\delta) \\
&\quad:= 
\inf_{(\mathbf{A}_i)_{i=1}^n\in\prod_{i=1}^n\Gamma_R(\tau_i;N,m,\delta)}
\log\int_{U(N)^n}d\gamma_{\mathrm{U}(N)}^{\otimes n}
\exp\bigl(-N^2\mathrm{tr}_N(h((V_i\mathbf{A}_iV_i^*)_{i=1}^n))\bigr), \\
&\pi_{\mathrm{orb},R}(h:(\tau_i)_{i=1}^n)
=\sup_{m \in \mathbb{N}\atop\delta>0}\,\limsup_{N\to\infty}{1\over N^2}\,\underline{\pi}_{\mathrm{orb},R}(h:(\tau_i)_{i=1}^n;N,m,\delta).
\end{align*}
\end{remark}

\section{Orbital $\eta$-entropy} 

We keep the notations in \S2. Let $\tau_i \in TS(\mathcal{C}_R(\mathbf{x}_i))$,
$1 \leq i \leq n$, be given. By Proposition \ref{P2.2} we have known that
$h\in\mathcal{C}_R(\mathbf{x})^{sa} \mapsto \pi_{\mathrm{orb},R}(h:(\tau_i)_{i=1}^n)$ is
convex and norm-continuous (as long as all $\tau_i$ have f.d.a.). Thus we may consider
its Legendre transform with respect to the (real) Banach space duality between
$\mathcal{C}_R(\mathbf{x})^{sa}$ and $\mathcal{C}_R(\mathbf{x})^{* sa}$, the self-adjoint
part of the dual space $\mathcal{C}_R(\mathbf{x})^*$

\begin{definition}\label{D3.1}
For a given $\varphi \in \mathcal{C}_R(\mathbf{x})^{* sa}$ we define 
$$
\eta_{\mathrm{orb},R}(\varphi:(\tau_i)_{i=1}^n)
:= \inf\{ \varphi(h) + \pi_{\mathrm{orb},R}(h:(\tau_i)_{i=1}^n)\,|\,h \in \mathcal{C}_R(\mathbf{x})^{sa}\}\ (\in[-\infty,\infty)). 
$$
\end{definition}

Remark that this is not exactly the Legendre transform; it is indeed the minus Legendre transform of $h \mapsto \pi_{\mathrm{orb},R}(h:(\tau_i)_{i=1}^n)$.

\begin{proposition}\label{P3.1} If $\eta_{\mathrm{orb},R}(\varphi:(\tau_i)_{i=1}^n) > -\infty$, then $\varphi$ must be in $TS(\mathcal{C}_R(\mathbf{x}))$ and satisfy $\varphi\!\upharpoonright_{\mathcal{C}_R(\mathbf{x}_i)} = \tau_i$ for every $1 \leq i \leq n$.
\end{proposition} 
\begin{proof} The same proof as that of \cite[Lemma 3.3]{Hiai:CMP05} works well to show that the finiteness assumption implies $\varphi \in TS(\mathcal{C}_R(\mathbf{x}))$. Hence it suffices to prove that if $\tau \in TS(\mathcal{C}_R(\mathbf{x}))$ satisfies $\tau\!\upharpoonright_{\mathcal{C}_R(\mathbf{x}_i)} \neq \tau_i$ for some $1 \leq i \leq n$, then $\eta_{\mathrm{orb},R}(\tau:(\tau_i)_{i=1}^n) = -\infty$. We may and do assume that all $\tau_i$ have f.d.a.; otherwise, $\pi_{\mathrm{orb},R}(h:(\tau_i)_{i=1}^n) = -\infty$ for every $h \in \mathcal{C}_R(\mathbf{x})^{sa}$ by definition and Lemma \ref{L2.1}.

By assumption we have $\tau(p) \neq \tau_{i_0}(p)$ for some $1 \leq i_0 \leq n$ and some non-commutative polynomial $p=p^*$ in $\mathbf{x}_{i_0}$. We may and do assume that $\tau(p) < \tau_{i_0}(p)$. Let $\varepsilon>0$ be arbitrarily given. Then there are $m \in \mathbb{N}$ and $\delta > 0$ such that $\mathbf{A} \in \Gamma_R(\tau_{i_0}\,;N,m,\delta)$ implies $|\mathrm{tr}_N(p(\mathbf{A})) - \tau_{i_0}(p)| < \varepsilon$ so that $\tau_{i_0}(p) - \varepsilon < \mathrm{tr}_N(p(\mathbf{A}))$. By the f.d.a.~assumption with Lemma \ref{L2.1}, for every sufficiently large $N$ we have known $\Gamma(\tau_{i_0}\,;N,m,\delta) \neq \emptyset$ and moreover $\tau_{i_0}(p) - \varepsilon \leq \inf_{\mathbf{A} \in \Gamma(\tau_{i_0}\,;N,m,\delta)} \mathrm{tr}_N(p(\mathbf{A}))$. 
For such $N$ and for any $\alpha > 0$ we then have 
\begin{align*} 
&\pi_{\mathrm{orb},R}(\alpha p:(\tau_i)_{i=1}^n\,;N,m,\delta) \\
&\quad= 
\sup_{\mathbf{A} \in \Gamma(\tau_{i_0}\,;N,m,\delta)} \log \int_{\mathrm{U}(N)} d\gamma_{\mathrm{U}(N)}(V) \exp\bigl(-N^2\mathrm{tr}_N(\alpha p(V \mathbf{A} V^*))\bigr) \\
&\quad\le
-N^2 \inf_{\mathbf{A} \in \Gamma(\tau_{i_0}\,;N,m,\delta)}\mathrm{tr}_N(\alpha p(\mathbf{A})) 
\leq -N^2 \alpha(\tau_{i_0}(p) - \varepsilon). 
\end{align*}
Therefore, $\tau(\alpha p) + \pi_{\mathrm{orb},R}(\alpha p:(\tau_i)_{i=1}^n)
\leq \alpha(\tau(p) - \tau_{i_0}(p)) + \alpha \varepsilon$. Since $\varepsilon>0$ is arbitrary, $\tau(\alpha p) + \pi_{\mathrm{orb},R}(\alpha p:(\tau_i)_{i=1}^n) \leq \alpha(\tau(p) - \tau_{i_0}(p))$, which immediately implies the desired assertion by letting $\alpha \to \infty$. 
\end{proof}

\begin{corollary}\label{C3.2} For every $h \in \mathcal{C}_R(\mathbf{x})^{sa}$ we have 
\begin{align*}
&\pi_{\mathrm{orb},R}(h:(\tau_i)_{i=1}^n) \\
&\quad= \max\{-\tau(h) + \eta_{\mathrm{orb},R}(\tau:(\tau_i)_{i=1}^n)\,|\,
\tau \in TS(\mathcal{C}_R(\mathbf{x})),\,
\tau\!\upharpoonright_{\mathcal{C}_R(\mathbf{x}_i)} = \tau_i,\,1\le i\le n\}.
\end{align*}
\end{corollary}
\begin{proof} The Legendre transform duality (see e.g., \cite[\S I.6]{Simon:BookPhys}) gives
$$
\pi_{\mathrm{orb},R}(h:(\tau_i)_{i=1}^n)
= \sup\{ -\varphi(h) + \eta_{\mathrm{orb},R}(\varphi:(\tau_i)_{i=1}^n)\,|\,\varphi \in \mathcal{C}_R(\mathbf{x})^{* sa}\}.
$$
By Proposition \ref{P3.1} the above supremum can be taken over tracial states restricted
as asserted. Then supremum can be replaced with maximum since
$\varphi\mapsto\eta_{\mathrm{orb},R}(\varphi:(\tau_i)_{i=1}^n)$ is upper
semicontinuous in the weak* topology by definition and the set of $\tau\in TS(\mathcal{C}_R(\mathbf{x}))$ with
$\tau\!\upharpoonright_{\mathcal{C}_R(\mathbf{x}_i)} = \tau_i$ for $1\le i\le n$ is
weak*-compact.
\end{proof}

For each $\tau\in TS(\mathcal{C}_R(\mathbf{x}))$, letting
$\tau_i:=\tau\!\upharpoonright_{\mathcal{C}_R(\mathbf{x}_i)}$, $1\le i\le n$, we write
$$
\eta_{\mathrm{orb},R}(\tau):=\eta_{\mathrm{orb},R}(\tau:(\tau_i)_{i=1}^n)
$$
and call it the \emph{orbital $\eta$-entropy} of $\tau$ (relative to the formation
\eqref{F1}).

\begin{proposition}\label{P3.3}
The orbital $\eta$-entropy $\eta_{\mathrm{orb},R}(\tau)$ for
$\tau\in TS(\mathcal{C}_R(\mathbf{x}))$ enjoys the following properties{\rm:}
\begin{itemize}
\item[(1)] $\eta_{\mathrm{orb},R}(\tau)\le0$.
\item[(2)] $\eta_{\mathrm{orb},R}(\tau)=-\infty$ if one of the $\tau_i$ does not have
f.d.a.
\item[(3)] For $1\le n'<n$,
$$
\eta_{\mathrm{orb},R}(\tau)\le\eta_{\mathrm{orb},R}
\bigl(\tau\!\upharpoonright_{\mathcal{C}_R(\sqcup_{i=1}^{n'}\mathbf{x}_i)}\bigr)
+\eta_{\mathrm{orb},R}
\bigl(\tau\!\upharpoonright_{\mathcal{C}_R(\sqcup_{i=n'+1}^n\mathbf{x}_i)}\bigr).
$$
\item[(4)] $\tau \mapsto \eta_{\mathrm{orb},R}(\tau)$ is upper semicontinuous on
$TS(\mathcal{C}(\mathbf{x}))$ equipped with the weak* topology.
\item[(5)] For given $\tau_i^0\in TS(\mathcal{C}_R(\mathbf{x}_i))$, $1\le i\le n$,
$\tau \mapsto \eta_{\mathrm{orb},R}(\tau)$ is concave on
$$
\{\tau\in TS(\mathcal{C}_R(\mathbf{x}))
\,|\,\tau\!\upharpoonright_{\mathcal{C}_R(\mathbf{x}_i)}=\tau_i^0,\,1\le i\le n\}.
$$
\end{itemize}
\end{proposition}

\begin{proof}
(1), (2) and (5) are obvious by definition and Proposition \ref{P2.2}\,(1).

(3) For every $h_1 \in \mathcal{C}_R(\bigsqcup_{i=1}^{n'}\mathbf{x}_i)^{sa}$ and
$h_2 \in \mathcal{C}_R(\bigsqcup_{i=n'+1}^n\mathbf{x}_i)^{sa}$ one has
\begin{align*} 
\eta_{\mathrm{orb},R}(\tau) 
&\leq \tau(h_1+h_2) + \pi_{\mathrm{orb},R}(h_1+h_2:(\tau_i)_{i=1}^n) \\
&\leq \tau(h_1) + \tau(h_2) + \pi_{\mathrm{orb},R}(h_1:(\tau_i)_{i=1}^{n'}) + 
\pi_{\mathrm{orb},R}(h_2:(\tau_i)_{i=n'+1}^n)
\end{align*}
by Proposition \ref{P2.2}\,(5), which gives the desired assertion.

(4) Since the weak* topology on $TS(\mathcal{C}_R(\mathbf{x}))$ is metrizable, we may
only consider a sequence $\tau^{(k)}\in TS(\mathcal{C}_R(\mathbf{x}))$, $k\in\mathbb{N}$,
such that $\tau^{(k)}\to\tau$ in the weak* topology as $k\to\infty$. For each $1\le i\le n$ let
$\tau_i^{(k)}:=\tau^{(k)}\!\upharpoonright_{\mathcal{C}_R(\mathbf{x}_i)}$ and
$\tau_i:=\tau\!\upharpoonright_{\mathcal{C}_R(\mathbf{x}_i)}$, $1\le i\le n$; then
$\tau_i^{(k)}\to\tau_i$ in the weak* topology as $k\to\infty$. For every $h\in\mathcal{C}_R(\mathbf{x})^{sa}$, by Lemma \ref{L2.3} we get
\begin{align*}
\tau(h)+\pi_{\mathrm{orb},R}(h:(\tau_i)_{i=1}^n)
%&
\ge\limsup_{k\to\infty}\bigl\{\tau^{(k)}(h)
+\pi_{\mathrm{orb},R}(h:(\tau_i^{(k)})_{i=1}^n)\bigr\}
%\\&
\ge\limsup_{k\to\infty}\eta_{\mathrm{orb},R}(\tau^{(k)}),
\end{align*}
which gives
$\eta_{\mathrm{orb},R}(\tau)\ge\limsup_{k\to\infty}\eta_{\mathrm{orb},R}(\tau^{(k)})$.
\end{proof}

In addition to $\mathbf{x}_i=(x_{ij})_{j=1}^{r(i)}$ and $R>0$ we consider
multi-indeterminates $\mathbf{y}_i = (y_{ij})_{j=1}^{s(i)}$, $1 \leq i \leq n$, and
$S > 0$. We define $\mathcal{C}_S(\mathbf{y}_i)$, $1 \leq i \leq n$, and
$\mathcal{C}_S(\mathbf{y})$ with $\mathbf{y} := \mathbf{y}_1\sqcup\cdots\sqcup\mathbf{y}_n$
as before. Let $(\mathcal{M},\tau)$ be a tracial $W^*$-probability space, that is,
$\mathcal{M}$ is a finite von Neumann algebra and $\tau$ a faithful normal tracial state on
$\mathcal{M}$. We write $\Vert X\Vert_\infty$ for the operator norm of $X\in\mathcal{M}$.
The next theorem will play a key r\^{o}le in the rest of this section. Indeed, it enables us to define a kind of free mutual information by means of orbital $\eta$-entropy.

\begin{theorem}\label{T3.4}
For $1\le i\le n$ let $\mathbf{X}_i=(X_{ij})_{j=1}^{r(i)}$ and
$\mathbf{Y}_i=(Y_{ij})_{j=1}^{s(i)}$ be self-adjoint random
multi-variables in $\mathcal{M}^{sa}$ such that $\|X_{ij}\|_\infty\le R$ and
$\|Y_{ij}\|_\infty\le S$ for all $i,j$. Let
$\tau_\mathbf{X}^{(R)}\in TS(\mathcal{C}_R(\mathbf{x}))$ be the tracial state induced from
$\tau$ via the $*$-homomorphism determined by
$x_{ij}\in\mathcal{C}_R(\mathbf{x})\mapsto X_{ij}\in\mathcal{M}$, and
$\tau_\mathbf{Y}^{(S)}\in TS(\mathcal{C}_S(\mathbf{y}))$ be similarly induced via
$y_{ij}\in\mathcal{C}_S(\mathbf{y})\mapsto Y_{ij}\in\mathcal{M}$. If
$\mathbf{Y}_i\subset W^*(\mathbf{X}_i)$ in $\mathcal{M}$ for every $1\le i\le n$, then $\eta_{\mathrm{orb},R}(\tau_\mathbf{X}^{(R)})
\le\eta_{\mathrm{orb},S}(\tau_\mathbf{Y}^{(S)})$. 
\end{theorem}
\begin{proof} Firstly, assume that $\max_{i,j}\|Y_{ij}\|_\infty<S$, and let
$\alpha > \eta_{\mathrm{orb},S}(\tau_\mathbf{Y}^{(S)})$ be arbitrary. By definition there exists a non-commutative polynomial $q=q^*$ in $\mathbf{y}$ such that 
\begin{equation}\label{F3}
\tau(q(\mathbf{Y})) + \pi_{\mathrm{orb},S}(q:(\tau_{\mathbf{Y}_i}^{(S)})_{i=1}^n)
< \alpha, 
\end{equation}
where $\tau_{\mathbf{Y}_i}^{(S)}\in TS(\mathcal{C}_S(\mathbf{y}_i))$ induced by
$y_{ij}\mapsto Y_{ij}$ coincides with
$\tau_\mathbf{Y}^{(S)}\!\upharpoonright_{\mathcal{C}_R(\mathbf{y}_i)}$.
We may and do assume that all $\mathbf{X}_i$ (equivalently, all
$\tau_{\mathbf{X}_i}^{(R)}$) have f.d.a.; otherwise $\eta_{\mathrm{orb},R}(\tau_\mathbf{X}^{(R)}) = -\infty$ and the desired assertion trivially holds. Let $m \in \mathbb{N}$ and $\delta>0$ be arbitrarily given. Then, by the proof of \cite[Lemma 2.3]{BelinschiBercovici:PacificJMath03} (or the proof of \cite[Lemma 2.3]{Ueda:IUMJ1x}) there exist $m' \in \mathbb{N}$ and $\delta'>0$ so that 
\begin{equation}\label{F4}
(B_j)_{j=1}^{s(i)} \in \Gamma_\infty(\mathbf{Y}_i\,;N,m',\delta') \Longrightarrow 
(f_S(B_j))_{j=1}^{s(i)} \in \Gamma_S(\mathbf{Y}_i\,;N,m,\delta)
\end{equation}
holds for every $N \in \mathbb{N}$ and $1 \leq i \leq n$, where $f_S(t) := Sf(t/S)$ with
the continuous function $f : \mathbb{R} \to [-1,1]$ defined by
$f(t) = t$ for $-1 \leq t \leq 1$, $f(t) = -1$ for $t < -1$ and $f(t) = 1$ for $t > 1$.
For every $\varepsilon>0$, by the Kaplansky density theorem one can choose non-commutative polynomials $p_{ij} = p_{ij}^*$ in $\mathbf{x}_i$ for $1 \leq i \leq n$ and
$1 \leq j \leq s(i)$ in such a way that 
\begin{align} 
&\Vert p_{ij}(\mathbf{X}_i) \Vert_\infty 
\leq \Vert Y_{ij}\Vert_\infty < S, \label{F5}\\
&\Vert p_{ij}(\mathbf{X}_i) -Y_{ij} \Vert_{2,\tau} 
< \delta'/(2m' S^{m'-1}), \label{F6}\\
&\big|\tau\bigl(q\bigl(((p_{ij}(\mathbf{X}_i))_{j=1}^{s(i)})_{i=1}^n\bigr)\bigr)
- \tau(q(\mathbf{Y}))\big| < \varepsilon, \label{F7}
\end{align}
where $\|X\|_{2,\tau}:=\tau(X^*X)^{1/2}$, the $2$-norm on $\mathcal{M}$.
Looking at the $p_{ij}$ one can also find $m'' \in \mathbb{N}$ and $\delta''>0$ in such a
way that for every $1 \leq i \leq n$ and $N \in \mathbb{N}$ one has 
$$
\mathbf{A} \in \Gamma_R(\mathbf{X}_i\,;N,m'',\delta'') \Longrightarrow
\big|\mathrm{tr}_N(p_{i j_1}(\mathbf{A})\cdots p_{ij_l}(\mathbf{A}))
-\tau(p_{ij_1}(\mathbf{X}_i)\cdots p_{ij_l}(\mathbf{X}_i))\big| < \delta'/2
$$
whenever $1 \leq j_k \leq r(i)$, $1\leq k \leq l$ and $1 \leq l \leq m'$. (See the proof of
\cite[Theorem 2.6\,(6)]{Ueda:IUMJ1x}.) Hence, if
$\mathbf{A} \in \Gamma_R(\mathbf{X}_i\,;N,m'',\delta'')$, then by \eqref{F5} and \eqref{F6}
one has
\begin{align*}
&\big|\mathrm{tr}_N(p_{i j_1}(\mathbf{A})\cdots p_{ij_l}(\mathbf{A}))
- \tau(Y_{ij_1}\cdots Y_{ij_l})\big| \\
&\quad\leq 
\big|\mathrm{tr}_N(p_{i j_1}(\mathbf{A})\cdots p_{ij_l}(\mathbf{A}))-\tau(p_{ij_1}(\mathbf{X}_i)\cdots p_{ij_l}(\mathbf{X}_i))\big| \\
&\qquad\qquad+ 
\big|\tau(p_{ij_1}(\mathbf{X}_i)\cdots p_{ij_l}(\mathbf{X}_i))
-\tau(Y_{ij_1}\cdots Y_{ij_l})\big| < \delta'
\end{align*}
whenever $1 \leq j_k \leq s(i)$, $1 \leq k \leq l$ and $1 \leq l \leq m'$, so that
$(p_{ij}(\mathbf{A}))_{j=1}^{s(i)} \in \Gamma_\infty(\mathbf{Y}_i\,;N,m',\delta')$.
Thanks to \eqref{F4} we have shown that
$$
\mathbf{A} \in \Gamma_R(\mathbf{X}_i\,;N,m'',\delta'') \Longrightarrow
(f_S(p_{ij}(\mathbf{A})))_{j=1}^{s(i)} \in \Gamma_S(\mathbf{Y}_i\,;N,m,\delta)
$$
for every $1\le i\le n$. Set $h := q\bigl(\bigl((f_S(p_{ij}(\mathbf{x}_i)))_{j=1}^{s(i)}\bigr)_{i=1}^n\bigr)
\in \mathcal{C}_R(\mathbf{x})^{sa}$, where $f_S(p_{ij}(\mathbf{x}_i))\in\mathcal{C}_R(\mathbf{x}_i)^{sa}$ is defined via
continuous functional calculus of
$p_{ij}(\mathbf{x}_i)\in\mathcal{C}_R(\mathbf{x}_i)^{sa}$ by a continuous function $f_S$.
Note that for $\mathbf{A}_i \in \Gamma_R(\mathbf{X}_i\,;N,m'',\delta'')$ and
$V_i\in\mathrm{U}(N)$, $1\le i\le n$, we have
$$
h((V_i\mathbf{A}_iV_i^*)_{i=1}^n)
=q\bigl(\bigl((V_if_S(p_{ij}(\mathbf{A}_i))V_i^*)_{j=1}^{s(i)}
\bigr)_{i=1}^n\bigr)
=q((V_i\mathbf{B}_iV_i^*)_{i=1}^n)
$$
for some $\mathbf{B}_i\in\Gamma_S(\mathbf{Y}_i\,;N,m,\delta)$, $1\le i\le n$. Therefore,
\begin{align*}
&\pi_{\mathrm{orb},R}(h:(\tau_{\mathbf{X}_i}^{(R)})_{i=1}^n) %\\&\quad
\le\pi_{\mathrm{orb},R}(h:(\tau_{\mathbf{X}_i}^{(R)})_{i=1}^n\,;N,m'',\delta'') \\
&\quad= 
\sup_{\mathbf{A}_i \in \Gamma_R(\mathbf{X}_i\,;N,m",\delta") \atop 1 \leq i \leq n} 
\log \int_{\mathrm{U}(N)^n} d\gamma_{\mathrm{U}(N)}^{\otimes n}(V_i)
\exp\bigl(-N^2\mathrm{tr}_N\bigl(h(V_i\mathbf{A}_iV_i^*)_{i=1}^n)\bigr)\bigr) \\
&\quad\leq
\sup_{\mathbf{B}_i \in \Gamma_S(\mathbf{Y}_i\,;N,m,\delta) \atop 1 \leq i \leq n} 
\log \int_{\mathrm{U}(N)^n} d\gamma_{\mathrm{U}(N)}^{\otimes n}(V_i)
\exp\bigl(-N^2\mathrm{tr}_N\bigl(q((V_i \mathbf{B}_i V_i^*)_{i=1}^n)\bigr)\bigr) \\
&\quad= \pi_{\mathrm{orb},S}(q:(\tau_{\mathbf{Y}_i}^{(S)})_{i=1}^n\,;N,m,\delta).
\end{align*}
Furthermore, since $f_S(p_{ij}(\mathbf{X}_i)) = p_{ij}(\mathbf{X}_i)$ by \eqref{F5},
we get $\tau(h(\mathbf{X}))<\tau(q(\mathbf{Y}))+\varepsilon$ thanks to \eqref{F7}. We
thus obtain
\begin{align*} 
\eta_{\mathrm{orb},R}(\tau_\mathbf{X}^{(R)}) 
&\leq 
\tau(h(\mathbf{X})) + \pi_{\mathrm{orb},R}(h:(\tau_{\mathbf{X}_i}^{(R)})_{i=1}^n) \\
&\leq 
\tau(q(\mathbf{Y})) + \varepsilon 
+ \pi_{\mathrm{orb},S}(q:(\tau_{\mathbf{Y}_i}^{(S)})_{i=1}^n\,;m,\delta).
\end{align*} 
Since $m \in \mathbb{N}$ and $\delta, \varepsilon>0$ are arbitrary, one has $$\eta_{\mathrm{orb},R}(\tau_\mathbf{X}^{(R)})
\leq \tau(q(\mathbf{Y})) + \pi_{\mathrm{orb},S}(q:(\tau_{\mathbf{Y}_i}^{(S)})_{i=1}^n)
< \alpha$$
thanks to \eqref{F3}. This implies the desired inequality when
$\max_{i,j}\|Y_{ij}\|_\infty<S$.

Finally, when $\max_{i,j}\|Y_{ij}\|_\infty\le S$, as in the proof of
\cite[Corollary 2.7]{Ueda:IUMJ1x} we replace $Y_{ij}$ with $rY_{ij}$ where $0<r<1$,
so that $\eta_{\mathrm{orb},R}(\tau_\mathbf{X}^{(R)})\le
\eta_{\mathrm{orb},S}(\tau_{r\mathbf{Y}}^{(S)})$ follows from the case that already proved.
Since it is obvious that $\tau_{r\mathbf{Y}}^{(S)}\to\tau_\mathbf{Y}^{(S)}$ in the weak* topology as
$r\nearrow1$, it follows from Proposition \ref{P3.3}\,(4) that
$\eta_{\mathrm{orb},R}(\tau_\mathbf{X}^{(R)})\le
\limsup_{r\nearrow1}\eta_{\mathrm{orb},R}(\tau_{r\mathbf{Y}}^{(S)})
\le\eta_{\mathrm{orb},S}(\tau_\mathbf{Y}^{(S)})$. 
\end{proof}  

In the particular case where $\mathbf{X}_i=\mathbf{Y}_i$ for every $1\le i\le n$, the
above theorem says that $\eta_{\mathrm{orb},R}(\tau_\mathbf{X}^{(R)})$ is independent
of the choice of $R\ge\max_{i,j}\|X_{ij}\|_\infty$, and thus the next definition is
indeed justified.

\begin{definition}\label{D3.1}
Let $(\mathcal{M},\tau)$ be a $W^*$-probability space. For self-adjoint random multi-variables
$\mathbf{X}_i$, $1 \leq i \leq n$, in $\mathcal{M}$ we define 
$$
\eta_\mathrm{orb}(\mathbf{X}_1,\dots,\mathbf{X}_n)
:= \eta_{\mathrm{orb},R}(\tau_\mathbf{X}^{(R)}),
$$ 
where $\mathbf{X} := \mathbf{X}_1\sqcup\cdots\sqcup\mathbf{X}_n$ and
$R \geq \max_{i,j}\Vert X_{ij}\Vert_\infty$, and call it the \emph{orbital $\eta$-entropy}
of $\mathbf{X}_1,\dots,\mathbf{X}_n$. 
\end{definition}

Basic properties of $\eta_\mathrm{orb}$ for self-adjoint random multi-variables are summarized as follows:

\begin{theorem}\label{T3.5} The orbital $\eta$-entropy
$\eta_\mathrm{orb}(\mathbf{X}_1,\dots,\mathbf{X}_n)$ of self-adjoint random multi-variables
$\mathbf{X}_1,\dots,\mathbf{X}_n$ in $\mathcal{M}$ enjoys the following
properties{\rm:}
\begin{itemize}
\item[(1)] $\eta_\mathrm{orb}(\mathbf{X}_1,\dots,\mathbf{X}_n) \leq 0$.
\item[(2)] $\eta_\mathrm{orb}(\mathbf{X}_1,\dots,\mathbf{X}_n) = -\infty$ if one of the
$\mathbf{X}_i$ does not have f.d.a. 
\item[(3)] $\eta_\mathrm{orb}(\mathbf{X}_1,\dots,\mathbf{X}_n) \leq \eta_\mathrm{orb}(\mathbf{X}_1,\dots,\mathbf{X}_{n'}) + \eta_\mathrm{orb}(\mathbf{X}_{n'+1},\dots,\mathbf{X}_n)$ for $1\le n'<n$. 
\item[(4)] If $\mathbf{X}_i=(X_{ij})_{j=1}^{r(i)}$ and
$\mathbf{X}_i^{(k)}=(X_{ij}^{(k)})_{j=1}^{r(i)}$ are in $\mathcal{M}^{sa}$ for
$1\le i\le n$ and $k\in\mathbb{N}$ and $X_{ij}^{(k)}\to X_{ij}$ strongly as $k\to\infty$
for every $i,j$, then
$$
\eta_\mathrm{orb}(\mathbf{X}_1,\dots,\mathbf{X}_n) \geq 
\limsup_{k\to\infty}\eta_\mathrm{orb}(\mathbf{X}_1^{(k)},\dots,\mathbf{X}_n^{(k)}). 
$$
\item[(5)] If $\mathbf{Y}_i \subset W^*(\mathbf{X}_i)$ in $\mathcal{M}$ for every $1 \leq i \leq n$, then we have $\eta_\mathrm{orb}(\mathbf{X}_1,\dots,\mathbf{X}_n) \leq \eta_\mathrm{orb}(\mathbf{Y}_1,\dots,\mathbf{Y}_n)$. In particular, $\eta_\mathrm{orb}(\mathbf{X}_1,\dots,\mathbf{X}_n)$ depends only on the von Neumann subalgebras $W^*(\mathbf{X}_i)$ generated by $\mathbf{X}_i$ in $\mathcal{M}$.
\item[(6)] $\chi_\mathrm{orb}(\mathbf{X}_1,\dots,\mathbf{X}_n) \leq \eta_\mathrm{orb}(\mathbf{X}_1,\dots,\mathbf{X}_n)$ holds in general, where
$\chi_\mathrm{orb}(\mathbf{X}_1,\dots,\mathbf{X}_n)$ is the orbital free entropy
introduced in \cite{HiaiMiyamotoUeda:IJM09,Ueda:IUMJ1x}.
\item[(7)] $\eta_\mathrm{orb}(\mathbf{X}) = 0$ if $\mathbf{X}$ has f.d.a.{\rm;} otherwise $-\infty$. 
\item[(8)] If $\mathbf{X}_1,\dots,\mathbf{X}_n$ are freely independent and each $\mathbf{X}_i$ has f.d.a., then $\eta_\mathrm{orb}(\mathbf{X}_1,\dots,\mathbf{X}_n) = 0$. 
\end{itemize}
\end{theorem}

The converse of the above (8) also holds, but its proof needs the notion of `orbital equilibrium' tracial states; hence we postpone the complete assertion to \S5.

\begin{proof}
(1)--(3) are obvious from Proposition \ref{P3.3}\,(1)--(3).

(4) Since $\sup_k\|X_{ij}^{(k)}\|_\infty<\infty$ for each $i,j$, we can choose $R>0$ so
that $\|X_{ij}^{(k)}\|_\infty\le R$ for all $i,j,k$. Let
$\tau_\mathbf{X}^{(R)},\tau_{\mathbf{X}^{(k)}}^{(R)}\in TS(\mathcal{C}_R(\mathbf{x}))$ be
defined as in Theorem \ref{T3.4}. Then the weak* convergence
$\tau_{\mathbf{X}^{(k)}}^{(R)}\to\tau_\mathbf{X}^{(R)}$ is an immediate consequence of
the strong convergence $X_{ij}^{(k)}\to X_{ij}$ for each $i,j$. Hence the result follows
from Proposition \ref{P3.3}\,(4).

(5) immediately follows from Theorem \ref{T3.4}. 

(6) We may and do assume that $\chi_\mathrm{orb}(\mathbf{X}_1,\dots,\mathbf{X}_n) > -\infty$; in particular, all $\mathbf{X}_i$ have f.d.a. Choose and fix a cut-off constant $R>0$ in such a way that $\max_{ij}\Vert X_{ij}\Vert_\infty \leq R$. Let $p=p^*$ be an arbitrary non-commutative polynomial in $\mathbf{x}$ and $\varepsilon>0$ be arbitrarily given. Then, looking at $p$ one can choose $m_0 \in \mathbb{N}$ and $\delta_0>0$ in such a way that, for any $N\in\mathbb{N}$ and any $(\mathbf{A})_{i=1}^n\in\prod_{i=1}^n(M_N(\mathbb{C})_R^{sa})^{r(i)}$, if $m \geq m_0$ and $0 < \delta \leq \delta_0$, then $(V_i)_{i=1}^n \in \Gamma_\mathrm{orb}(\mathbf{X}_1,\dots,\mathbf{X}_n:(\mathbf{A}_i)_{i=1}^n\,;N,m,\delta)$ implies 
$|\mathrm{tr}_N(p((V_i\mathbf{A}_i V_i^*)_{i=1}^n)) - \tau(p(\mathbf{X}))| < \varepsilon$ so that $-\mathrm{tr}_N(p((V_i\mathbf{A}_i V_i^*)_{i=1}^n)) > -(\tau(p(\mathbf{X}))+\varepsilon)$ holds. (See \cite[\S2]{Ueda:IUMJ1x} for the definition of
$\Gamma_\mathrm{orb}(\mathbf{X}_1,\dots,\mathbf{X}_n:(\mathbf{A}_i)_{i=1}^n\,;N,m,\delta)$.) Therefore, for any $N\in\mathbb{N}$ and any $\mathbf{A}_i \in \Gamma_R(\mathbf{X}_i\,;N,m,\delta)$, $1 \leq i \leq n$, with $m \geq m_0$ and $0 < \delta \leq \delta_0$ we get 
\begin{align*} 
&\int_{\mathrm{U}(N)^n} d\gamma_{\mathrm{U}(N)}^{\otimes n}(V_i) \exp\bigl(-N^2\mathrm{tr}_N(p((V_i\mathbf{A}_i V_i^*)_{i=1}^n))\bigr) \\
&\quad\geq 
\int_{\Gamma_\mathrm{orb}(\mathbf{X}_1,\dots,\mathbf{X}_n:(\mathbf{A}_i)_{i=1}^n\,;N,m,\delta)} d\gamma_{\mathrm{U}(N)}^{\otimes n}(V_i) \exp\bigl(-N^2\mathrm{tr}_N(p((V_i\mathbf{A}_i V_i^*)_{i=1}^n))\bigr) \\
&\quad> 
\exp\bigl(-N^2(\tau(p(\mathbf{X}))+\varepsilon)\bigr) \gamma_{\mathrm{U}(N)}^{\otimes n}\big(\Gamma_\mathrm{orb}(\mathbf{X}_1,\dots,\mathbf{X}_n:(\mathbf{A}_i)_{i=1}^n\,;N,m,\delta)\big), 
\end{align*}
and hence
$$
\pi_{\mathrm{orb},R}(p:(\tau_{\mathbf{X}_i}^{(R)})_{i=1}^n\,;N,m,\delta) 
\geq 
-N^2(\tau(p(\mathbf{X}))+\varepsilon) + \bar{\chi}_{\mathrm{orb},R}(\mathbf{X}_1,\dots,\mathbf{X}_n\,;N,m,\delta).
$$
By \cite[Proposition 2.4]{Ueda:IUMJ1x} (that contains the definition of
$\bar{\chi}_{\mathrm{orb},R}(\mathbf{X}_1,\dots,\mathbf{X}_n\,;N,m,\delta)$)
and \cite[Corollary 2.7]{Ueda:IUMJ1x} we get
$$
\pi_{\mathrm{orb},R}(p:(\tau_{\mathbf{X}_i}^{(R)})_{i=1}^n) 
\geq 
-\tau(p(\mathbf{X}))-\varepsilon + \chi_\mathrm{orb}(\mathbf{X}_1,\dots,\mathbf{X}_n).
$$
This yields the desired assertion.

(7) and (8) follow from the corresponding facts on $\chi_\mathrm{orb}$ \cite[Theorem 2.6\,(3),\,(8)]{Ueda:IUMJ1x} thanks to (6).      
\end{proof} 

\begin{remark}\label{R3.6} Similarly to the relation between $\chi$ and $\eta_R$ the equality $\chi_\mathrm{orb} = \eta_\mathrm{orb}$ in Theorem \ref{T3.5}\,(6) does not hold in general as follows. Choose two $n$-tuples $\mathbf{Y} = (Y_i)_{i=1}^n$, $\mathbf{Z} = (Z_i)_{i=1}^n$ of self-adjoint random variables in a tracial $W^*$-probability space in such a way that (i) $\mathbf{Y}$ is not a freely independent family but $\chi(\mathbf{Y}) > -\infty$, (ii) $\mathbf{Z}$ is a freely independent family, and (iii) $Y_i$ and $Z_i$ have the same distribution for every $1 \leq i \leq n$. Letting $R := \max_i\Vert Y_i\Vert_\infty \ (= \max_i\Vert Z_i\Vert_\infty)$ we obtain two tracial states $\tau_\mathbf{Y}^{(R)}, \tau_\mathbf{Z}^{(R)}$ on $\mathcal{C}_R(\mathbf{x})$ with $\mathbf{x} = (x_1)\sqcup\cdots\sqcup(x_n)$. Set $\tau := \frac{1}{2}(\tau_\mathbf{Y}^{(R)}+\tau_\mathbf{Z}^{(R)}) \in TS(\mathcal{C}_R(\mathbf{x}))$, and via the GNS representation associated with $\tau$ we obtain a new $n$-tuple $\mathbf{X} = (X_i)_{i=1}^n$ of self-adjoint random variables in a tracial $W^*$-probability space so that $\tau_\mathbf{X}^{(R)} = \tau$ holds; in particular, the distribution of $X_i$ coincides with those of $Y_i$ and $Z_i$ for every $1\leq i \leq n$. We have $\eta_\mathrm{orb}(X_1,\dots,X_n) = \eta_{\mathrm{orb},R}(\tau) \geq \frac{1}{2}(\eta_{\mathrm{orb},R}(\tau_\mathbf{Y}^{(R)}) + \eta_{\mathrm{orb},R}(\tau_\mathbf{Z}^{(R)})) \geq \frac{1}{2}(\chi_\mathrm{orb}(\mathbf{Y}) + \chi_\mathrm{orb}(\mathbf{Z})) > -\infty$ by Proposition \ref{P3.3}\,(5), Theorem \ref{T3.5}\,(6), \cite[Theorem 2.6]{HiaiMiyamotoUeda:IJM09} and $\chi(Z_i) = \chi(Y_i) > -\infty$ for all $1\leq i \leq n$. On the other hand, the degenerate convexity of $\chi$ \cite[$\chi.8$]{Voiculescu:Survey} implies that $\chi(\mathbf{X})=-\infty$. This implies $\chi_\mathrm{orb}(X_1,\dots,X_n)=-\infty$ thanks to \cite[Theorem 2.6]{HiaiMiyamotoUeda:IJM09} again, since $\chi(X_i) = \chi(Y_i) > -\infty$ for all $1\le i\le n$.
\end{remark}

\begin{remarks}\label{R3.7} (1) If the von Neumann algebra $W^*(\mathbf{X})$ is a factor, then $\tilde{\chi}_\mathrm{orb}(\mathbf{X}_1,\dots,\mathbf{X}_n)$, one of alternative approaches to the orbital free entropy due to Biane and Dabrowski \cite[\S7]{BianeDabrowski:AdvMath13}, agrees with $\chi_\mathrm{orb}(\mathbf{X}_1,\dots,\mathbf{X}_n)$ due to \cite[Theorem 7.3\,(6)]{BianeDabrowski:AdvMath13}. Hence $\tilde{\chi}_\mathrm{orb} \leq \eta_\mathrm{orb}$ holds too under the factoriality assumption. However, we do not know, at the present moment, whether or not the inequality holds without the assumption. It is worth noting that $\tilde\chi\le\eta_R$ holds in general
for two concavifications $\tilde\chi$ in \cite{BianeDabrowski:AdvMath13} and $\eta_R$ in \cite{Hiai:CMP05} of microstate free entropy, see \cite[Remark 4.5]{BianeDabrowski:AdvMath13}. 

(2) Since the finiteness of $\chi_\mathrm{orb}$ may be thought of as a kind of free independence, the known factoriality result on free product von Neumann algebras, see \cite[Theorem 4.1]{Ueda:AdvMath11}, suggest the following plausible conjecture: If $\chi_\mathrm{orb}(\mathbf{X}_1,\dots,\mathbf{X}_n) > -\infty$ and if at least one of the $W^*(\mathbf{X}_i)$ is diffuse, then $W^*(\mathbf{X}_1,\dots,\mathbf{X}_n)$ is a factor. Remark that a related result was given in \cite[Corollary 4.3]{DabrowskiIoana:Preprint12} based on Voiculescu's liberation theory.
\end{remarks}

\section{Orbital equilibrium tracial states}

In the notations as before let $\tau\in TS(\mathcal{C}_R(\mathbf{x}))$ and
$h\in\mathcal{C}_R(\mathbf{x})^{sa}$ be given. In \cite[p.238]{Hiai:CMP05} we said that $\tau$ is
an \emph{equilibrium} tracial state associated with $h$ if the equality
$\eta_R(\tau)=\tau(h)+\pi_R(h)$ holds. Its orbital analog is defined in a similar way.

\begin{definition}\label{D4.1}
We say that $\tau$ is an \emph{orbital equilibrium} tracial state associated with $h$ if
the equality
$$
\eta_{\mathrm{orb},R}(\tau)
= \tau(h) + \pi_{\mathrm{orb},R}(h:(\tau_i)_{i=1}^n)
$$
holds with finite value, where $\tau_i:=\tau\!\upharpoonright_{\mathcal{C}_R(\mathbf{x}_i)}$
for $1\le i\le n$. (Note that this restriction of the $\tau_i$ is indeed necessary for the above equality to hold with finite value, see Proposition \ref{P3.1}.)
\end{definition}

For any $h\in\mathcal{C}_R(\mathbf{x})^{sa}$, there is an equilibrium tracial state associated with $h$, and also given any $(\tau_i)_{i=1}^n$ with $\tau_i\in TS(\mathcal{C}_R(\mathbf{x}_i))$, $1\le i\le n$, there is an orbital equilibrium tracial state $\tau$ associated with $h$ such that $\tau\!\upharpoonright_{\mathcal{C}_R(\mathbf{x}_i)}=\tau_i$, $1\le i\le n$ (thanks to Corollary 3.2).
Moreover, a general theory ensures that, given any $(\tau_i)_{i=1}^n$ as above, the set of all $h \in\mathcal{C}_R(\mathbf{x})^{sa}$ for which there is a unique orbital equilibrium tracial state $\tau$ with $\tau\!\upharpoonright_{\mathcal{C}_R(\mathbf{x}_i)}=\tau_i$, $1\le i\le n$, forms a dense $G_\delta$-set. This is seen in the exactly same way as in the remark after \cite[Theorem 3.4]{Hiai:CMP05} (due to Proposition \ref{P2.2}\,(2),\,(4)).

\subsection{Notational conventions} 
Let $\tau \in TS(\mathcal{C}_R(\mathbf{x}))$ be given. Via the GNS representation associated with $\tau$ one obtains a finite von Neumann algebra $\mathcal{M}$ from $\mathcal{C}_R(\mathbf{x})$, and can `extend' $\tau$ to the whole $\mathcal{M}$ as a faithful normal tracial state. The original $x_{ij}$ give, via the representation, self-adjoint random variables $X_{ij} \in \mathcal{M}$, and set $\mathbf{X}_i := (X_{ij})_{j=1}^{r(i)}$, $1 \leq i \leq n$. We define the orbital free entropy $\chi_\mathrm{orb}(\tau)$ of the given $\tau$ to be $\chi_\mathrm{orb}(\mathbf{X}_1,\dots,\mathbf{X}_n)$. (Remark here that $\chi_\mathrm{orb}(\mathbf{X}_1,\dots,\mathbf{X}_n)=\chi_{\mathrm{orb},R}(\mathbf{X}_1,\dots,\mathbf{X}_n)$ for every $R \geq \max_{i,j}\Vert X_{ij}\Vert_\infty$ possibly with $R=\infty$, see \cite[Lemma 2.5, Corollary 2.7]{Ueda:IUMJ1x}, and the corresponding fact $\chi=\chi_R$ is due to  \cite{BelinschiBercovici:PacificJMath03} and \cite[Propositions 2.4, 2.6, 3.6 (b)]{Voiculescu:InventMath94} with a simple convergence argument.) We also write $\Gamma_\mathrm{orb}(\tau:\mathbf{A}\,;N,m,\delta)$ to be $\Gamma_\mathrm{orb}(\mathbf{X}_1,\dots,\mathbf{X}_n:\mathbf{A}\,;N,m,\delta)$ for a given $\mathbf{A}=(\mathbf{A}_i)_{i=1}^n \in \prod_{i=1}^n (M_N(\mathbf{C})^{sa})^{r(i)}$. Remark that $\eta_{\mathrm{orb},R}(\tau) = \eta_\mathrm{orb}(\mathbf{X}_1,\dots,\mathbf{X}_n)$ by definition, and hence $\chi_\mathrm{orb}(\tau) \leq \eta_{\mathrm{orb},R}(\tau)$ holds true by Theorem \ref{T3.5}\,(6).
   
\subsection{Criterion of orbital equilibrium} Let $\Xi(N) = (\Xi_i(N))_{i=1}^n$ with $\Xi_i(N) = (\xi_{ij}(N))_{j=1}^{r(i)} \in (M_N(\mathbb{C})_R^{sa})^{r(i)}$, $1 \leq i \leq n$, $N \in \mathbb{N}$, be a sequence of microstates, and assume that the tracial states $g \in \mathcal{C}_R(\mathbf{x}_i) \mapsto \mathrm{tr}_N(g(\Xi_i(N)))$ converges to some $\tau_i \in TS(\mathcal{C}_R(\mathbf{x}_i))$ in the weak* topology as $N\to\infty$ for every $1 \leq i \leq n$.  Fix $h \in \mathcal{C}_R(\mathbf{x})^{sa}$, and define the `\emph{orbital Gibbs micro-emsemble}' $\mu_N^{(h,\Xi(N))}$ on $\mathrm{U}(N)^n$ to be a probability measure
$$
\frac{1}{Z_N^{(h,\Xi(N))}}\exp(-N^2\mathrm{tr}_N(h((V_i\Xi_i(N)V_i^*)_{i=1}^n)))\,d\gamma_{\mathrm{U}(N)}^{\otimes n}(V_i),
$$
and also define the `\emph{orbital mean tracial state}' $\tau_N^{(h,\Xi(N))} \in TS(\mathcal{C}_R(\mathbf{x}))$ by
$$
\tau_N^{(h,\Xi(N))}(f):=\int_{\mathrm{U}(N)^n}
\,d\mu_N^{(h,\Xi(N))}(V_i)\,\mathrm{tr}_N(f((V_i\Xi_i(N)V_i^*)_{i=1}^n))),
\quad f\in\mathcal{C}_R(\mathbf{x}).
$$

\begin{proposition}\label{P4.2} 
Let $\tau \in TS(\mathcal{C}_R(\mathbf{x}))$ with $\tau\!\upharpoonright_{\mathcal{C}_R(\mathbf{x}_i)} = \tau_i$, $1 \leq i \leq n$, be given. Assume that every $C^*$-subalgebra $\mathcal{C}_R(\mathbf{x}_i)$ generates a hyperfinite von Neumann algebra in the GNS representation associated with $\tau$ and that 
$$
\lim_{N\to\infty}\frac{1}{N^2}\log\mu_N^{(h,\Xi(N))}(\Gamma_\mathrm{orb}(\tau:\Xi(N);N,m,\delta)) = 0
\leqno{\rm(\spadesuit)}
$$
for all sufficiently large $m \in \mathbb{N}$ and all sufficiently small $\delta > 0$. Then the tracial state $\tau$ must be orbital equilibrium associated with $h$, and moreover $\chi_\mathrm{orb}(\tau) = \eta_{\mathrm{orb},R}(\tau)$ holds. 
\end{proposition}
\begin{proof} Write $\Gamma(N,m,\delta) := \Gamma_\mathrm{orb}(\tau:\Xi(N);N,m,\delta)$ for simplicity. Let $\varepsilon>0$ be arbitrarily given. There is a non-commutative polynomial $p=p^*$ in $\mathbf{x}$ so that $\Vert p - h\Vert_R < \varepsilon/3$. Looking at $p$ one can choose $m \in \mathbb{N}$ and $\delta > 0$ so that  $|\mathrm{tr}_N(p((V_i\Xi_i(N)V_i^*)_{i=1}^n)) - \tau(p)| < \varepsilon/3$ whenever $(V_i)_{i=1}^n \in \Gamma(N,m,\delta)$ with arbitrary $N \in \mathbb{N}$. Hence, if $(V_i)_{i=1}^n \in \Gamma(N,m,\delta)$ with arbitrary $N \in \mathbb{N}$, then $|\mathrm{tr}_N(h((V_i\Xi_i(N)V_i^*)_{i=1}^n)) - \tau(h)| < \varepsilon$. Hence we have 
\begin{align*} 
&\frac{1}{N^2} \log\gamma_{\mathrm{U}(N)}^{\otimes n}(\Gamma(N,m,\delta)) \\
&\quad\underset{\varepsilon}{\approx} 
\tau(h) + \frac{1}{N^2}\log\int_{\Gamma(N,m,\delta)}d\gamma_{\mathrm{U}(N)}^{\otimes n}(V_i)\,\exp(-N^2\mathrm{tr}_N(h((V_i\Xi_i(N)V_i^*)_{i=1}^n))) \\
&\quad= 
\tau(h) + \frac{1}{N^2} \log Z_N^{(h,\Xi(N))} + \frac{1}{N^2}\log\mu_N^{(h,\Xi(N))}(\Gamma(N,m,\delta))
\end{align*}
for all $N \in \mathbb{N}$, where $a\underset{\varepsilon}{\approx}b$ means that
$|a-b|\le\varepsilon$. Hence, the assumption ($\spadesuit$) ensures that 
$$ 
\limsup_{N\to\infty}\frac{1}{N^2} \log\gamma_{\mathrm{U}(N)}^{\otimes n}(\Gamma(N,m,\delta)) 
\underset{\varepsilon}{\approx} 
\tau(h) + \limsup_{N\to\infty}\frac{1}{N^2} \log Z_N^{(h,\Xi(N))} 
$$
for all sufficiently large $m \in \mathbb{N}$ and all sufficiently small $\delta>0$. Taking the limit as $m\to\infty$, $\delta\searrow0$ we obtain $\chi_{\mathrm{orb}}(\tau) = \tau(h) + \pi_{\mathrm{orb},R}(h:(\tau_i)_{i=1}^n)$ by Proposition \ref{P2.4}, since $\varepsilon>0$ is arbitrary. The desired assertion immediately follows thanks to Theorem \ref{T3.5}\,(6) (see \S\S4.1). 
\end{proof}

\begin{remark}\label{R4.1} The assumption ($\spadesuit$) is satisfied when the `\emph{empirical orbital tracial state}' $f \in \mathcal{C}_R(\mathbf{x}) \mapsto \mathrm{tr}_N(f((V_i\Xi_i(N)V_i^*)_{i=1}^n))$ converges to $\tau$ in the weak* topology as $N\to\infty$, almost surely when $(V_i)_{i=1}^n \in \mathrm{U}(N)^n$ is distributed under $\mu_N^{(h,\Xi(N))}$. In fact, this implies a much stronger fact that 
$$
\lim_{N\to\infty}\mu_N^{(h,\Xi(N))}(\Gamma_\mathrm{orb}(\tau:\Xi(N);N,m,\delta)) = 1
$$ 
for every $m \in \mathbb{N}$ and $\delta > 0$. 
Hence the above proposition enables us to see that random matrix models studied by Collins, Guoionnet and Segala \cite{CollinsGuionnetSegala:AdvMath09} produce examples of orbital equilibrium tracial states in a suitable manner. However, the procedure of obtaining the desired tracial states is not so straightforward; hence we postpone its explanation to \S7 (especially Example \ref{Ex7.1}).    
\end{remark}

We do not know whether or not the assumption ($\spadesuit$) is sufficient for the equality between $\eta_\mathrm{orb}$ and Biane and Dabrowski's variant $\tilde{\chi}_\mathrm{orb}$. Similarly we do not yet know whether or not $\chi_\mathrm{orb} = \tilde\chi_\mathrm{orb}$ even for two projections, though the large deviation principle (that is apparently stronger than the convergence fact in Remark \ref{R4.1}) for random two projection matrices was already established in \cite{HiaiPetz:ACTA06}. For these questions further investigations seem necessary. Moreover, an interesting question arises in relation to the assumption ($\spadesuit$), see Remarks \ref{R4.6}\,(4).

\subsection{Connections to the free pressure and the $\eta$-entropy} Throughout this subsection, due to some technical difficulty, we assume that each given multi-indeterminate $\mathbf{x}_i$ consists of a single element;
namely $\mathbf{x}=(x_i)_{i=1}^n$ is an $n$ tuple of single indeterminates $x_i$. In this case, note that, for each $1\le i\le n$, $\mathcal{C}_R(x_i)=C[-R,R]$ and
$\tau_i\in TS(\mathcal{C}_R(x_i))$ may be regarded as a probability measure on $[-R,R]$.
Furthermore, remark that $\eta_R(\tau_i) = \chi(\tau_i)$ for every $\tau_i \in TS(\mathcal{C}_R(x_i))$, see \cite[Proposition 4.2]{Hiai:CMP05}.
For each $h\in\mathcal{C}_R(\mathbf{x})^{sa}$ and $N\in\mathbb{N}$ we define the
`\emph{Gibbs micro-ensemble}' $\lambda_{R,N}^h$ on $(M_N(\mathbb{C})_R^{sa})^n$ to be a
probability measure
$$
{1\over Z_{R,N}^h}\,\exp\bigl(-N^2\mathrm{tr}_N(h(\mathbf{A}))\bigr)
\mathbf{1}_{(M_N(\mathbb{C})_R^{sa})^n}(\mathbf{A})\,
d\Lambda_N^{\otimes n}(\mathbf{A}),  
$$
and also define the `\emph{mean tracial state}' $\tau_{R,N}^h\in TS(\mathcal{C}_R(\mathbf{x}))$ by
$$
\tau_{R,N}^h(f):=\int_{(M_N(\mathbb{C})_R^{sa})^n}
\,d\lambda_{R,N}^h(\mathbf{A})\,\mathrm{tr}_N(f(\mathbf{A})),
\quad f\in\mathcal{C}_R(\mathbf{x})
$$

\begin{proposition}\label{P4.3}
In the situation above, for every $h\in\mathcal{C}_R(\mathbf{x})^{sa}$ we have
$$
\pi_R(h)\ge\pi_{\mathrm{orb},R}(h:(\tau_i)_{i=1}^n)+\sum_{i=1}^n\chi(\tau_i),
$$
and the equality holds if $\lim_{N\to\infty}{1\over N^2}\log\lambda_{R,N}^h
\bigl(\prod_{i=1}^n\Gamma_R(\tau_i;N,m,\delta)\bigr)=0$ for all sufficiently large $m\in\mathbb{N}$ and all sufficiently small $\delta>0$.
\end{proposition}

\begin{proof}
We have
\begin{align*}
Z_{R,N}^h &\geq Z_{R,N}^h\lambda_{R,N}^h\Biggl(\prod_{i=1}^n\Gamma_R(\tau_i;N,m,\delta)\Biggr) \\
&=\int_{\prod_{i=1}^n\Gamma_R(\tau_i;N,m,\delta)}d\Lambda_N^{\otimes n}(A_i)\,
\exp\bigl(-N^2\mathrm{tr}_N(h((A_i)_{i=1}^n))\bigr) \\
&=\int_{\prod_{i=1}^n\Gamma_R(\tau_i;N,m,\delta)}d\Lambda_N^{\otimes n}(A_i)
\int_{\mathrm{U}(N)^n}d\gamma_{\mathrm{U}(N)}^{\otimes n}(V_i)\,
\exp\bigl(-N^2\mathrm{tr}_N(h((V_iA_iV_i^*)_{i=1}^n))\bigr) \\
&\ge\exp\bigl(\underline\pi_{\mathrm{orb},R}(h:(\tau_i)_{i=1}^n;N,m,\delta)\bigr)
\prod_{i=1}^n\Lambda_N\bigl(\Gamma_R(\tau_i;N,m,\delta)\bigr),
\end{align*}
where $\underline\pi_{\mathrm{orb},R}(h:(\tau_i)_{i=1}^n;N,m,\delta)$ is the quantity introduced in Remark \ref{R2.5}. The latter equality above is due to the unitary conjugation
invariance of $\Gamma_R(\tau_i;N,m,\delta)$ and $\Lambda_N$. Therefore,
\begin{align*}
&\limsup_{N\to\infty}\biggl({1\over N^2}\log Z_{R,N}^h+{n\over2}\log N\biggr) \\
&\quad\ge\limsup_{N\to\infty}{1\over N^2}\,
\underline\pi_{\mathrm{orb},R}(h:(\tau_i)_{i=1}^n;N,m,\delta)
+\sum_{i=1}^n\chi_R(\tau_i;m,\delta),
\end{align*}
since the limit $\chi_R(\tau_i;m,\delta):=\lim_{N\to\infty}\left({1\over N^2}
\log\Lambda_N\bigl(\Gamma_R(\tau_i;N,m,\delta)\bigr)+{1\over2}\log N\right)$
exists, see e.g., \cite[Theorem 5.6.2]{HiaiPetz:Book}. Letting $m\to\infty$ and $\delta\searrow0$ we get the
desired inequality thanks to Remark \ref{R2.5}. On the other hand, we also have
\begin{align*}
Z_{R,N}^h\lambda_{R,N}^h\Biggl(\prod_{i=1}^n\Gamma_R(\tau_i;N,m,\delta)\Biggr) 
\le\exp\bigl(\pi_{\mathrm{orb},R}(h:(\tau_i)_{i=1}^n;N,m,\delta)\bigr)
\prod_{i=1}^n\Lambda_N\bigl(\Gamma_R(\tau_i;N,m,\delta)\bigr)
\end{align*}
so that
\begin{align*}
&\limsup_{N\to\infty}\biggl({1\over N^2}\log Z_{R,N}^h+{n\over2}\log N\biggr)
+\liminf_{N\to\infty}{1\over N^2}
\log\lambda_{R,N}^h\Biggl(\prod_{i=1}^n\Gamma_R(\tau_i;N,m,\delta)\Biggr) \\
&\quad\le\limsup_{N\to\infty}{1\over N^2}\,
\pi_{\mathrm{orb},R}(h:(\tau_i)_{i=1}^n;N,m,\delta)
+\sum_{i=1}^n\chi_R(\tau_i;m,\delta).
\end{align*}
With the stated assumption this yields
$\pi_R(h)\le\pi_{\mathrm{orb},R}(h:(\tau_i)_{i=1}^n)+\sum_{i=1}^n\chi(\tau_i)$.
\end{proof}

\begin{corollary}\label{C4.4}
We have
$$
\eta_R(\tau)\ge\eta_{\mathrm{orb},R}(\tau)+\sum_{i=1}^n\chi(\tau_i)
\ge\chi_\mathrm{orb}(\tau)+\sum_{i=1}^n\chi(\tau_i)=\chi(\tau).
$$
Hence, if $\eta_R(\tau)=\chi(\tau)$ holds and $\chi(\tau_i)>-\infty$ for all $1\le i\le n$, then $\eta_{\mathrm{orb},R}(\tau)=\chi_\mathrm{orb}(\tau)$ holds. 
\end{corollary}

\begin{proof}
By Proposition \ref{P4.3}, for every
$h\in\mathcal{C}_R(\mathbf{x})^{sa}$ we have
$$
\tau(h)+\pi_R(h)
\ge\tau(h)+\pi_{\mathrm{orb},R}(h:(\tau_i)_{i=1}^n)+\sum_{i=1}^n\chi(\tau_i)
\ge\eta_{\mathrm{orb},R}(\tau)+\sum_{i=1}^n\chi(\tau_i),
$$
which implies the first inequality. The second inequality is contained in Theorem
\ref{T3.5}\,(6) and the last equality is \cite[Theorem 2.6]{HiaiMiyamotoUeda:IJM09}.
The latter assertion is immediate from the first.
\end{proof}

\begin{corollary}\label{C4.5}
If $\tau$ is an equilibrium tracial state associated with
$h\in\mathcal{C}_R(\mathbf{x})^{sa}$ and $\eta_R(\tau)=\chi(\tau)$, then $\tau$ is an
orbital equilibrium tracial state associated with $h$ and
$$
\pi_R(h)=\pi_{\mathrm{orb},R}(h:(\tau_i)_{i=1}^n)+\sum_{i=1}^n\chi(\tau_i), \qquad \eta_{\mathrm{orb},R}(\tau) = \chi_\mathrm{orb}(\tau).
$$
\end{corollary}

\begin{proof}
Since $\eta_R(\tau)$ is finite by the equilibrium assumption, we have
$\sum_{i=1}^n\chi(\tau_i) \geq \chi(\tau)=\eta_R(\tau)>-\infty$.  Therefore, we have
\begin{align*}
\eta_R(\tau)=\tau(h)+\pi_R(h)&\geq\tau(h)+\pi_{\mathrm{orb},R}(h:(\tau_i)_{i=1}^n)+\sum_{i=1}^n\chi(\tau_i) \\
&\ge\eta_{\mathrm{orb},R}(\tau)+\sum_{i=1}^n\chi(\tau_i)\geq\chi(\tau)=\eta_R(\tau)
\end{align*}
by Proposition \ref{P4.3} and Corollary \ref{C4.4}. Hence the desired assertions immediately follow.
\end{proof}

\begin{proposition}\label{P4.7} Let $h \in \mathcal{C}_R(\mathbf{x})^{sa}$ and $\tau \in TS(\mathcal{C}_R(\mathbf{x}))$ be given. Assume that 
$$
\lim_{N\to\infty}\frac{1}{N^2}\log\lambda_{R,N}^h(\Gamma_R(\tau;N,m,\delta)) = 0
\leqno{\rm(\clubsuit)}
$$
for all sufficiently large $m \in \mathbb{N}$ and all sufficiently small $\delta > 0$. Then the tracial state $\tau$ must be equilibrium associated with $h$, and $\chi(\tau) = \eta_R(\tau)$ holds. Moreover, the tracial state $\tau$ is also orbital equilibrium associated with $h$, and $
\pi_R(h) = \pi_{\mathrm{orb},R}(h) + \sum_{i=1}^n \chi(\tau_i)$ with $\tau_i := \tau\!\upharpoonright_{\mathcal{C}_R(\mathbf{x}_i)}$, $1 \leq i \leq n$, and $\chi_\mathrm{orb}(\tau) = \eta_{\mathrm{orb},R}(\tau)$ hold.  
\end{proposition}
\begin{proof} The proof is similar to that of Proposition \ref{P4.2}. Let $\varepsilon>0$ be arbitrarily chosen, and one can choose $m \in \mathbb{N}$ and $\delta>0$ so that $|\mathrm{tr}_N(h(\mathbf{A})) - \tau(h)| < \varepsilon$ whenever $\mathbf{A} \in \Gamma_R(\tau;N,m,\delta)$ with arbitrary $N \in \mathbb{N}$. Then we have 
\begin{align*} 
\frac{1}{N^2}\log\Lambda_N^{\otimes n}(\Gamma_R(\tau;N,m,\delta)) \underset{\varepsilon}{\approx} \tau(h) + \frac{1}{N^2}\log Z_{R,N}^h + \frac{1}{N^2}\log\lambda_{R,N}^h(\Gamma_R(\tau;N,m,\delta))
\end{align*}
as in the proof of Proposition \ref{P4.2}. Hence the assumption ($\clubsuit$) ensures that 
\begin{align*}
&\limsup_{N\to\infty}\biggl(\frac{1}{N^2}\log\Lambda_N^{\otimes n}(\Gamma_R(\tau;N,m,\delta))+\frac{n}{2}\log N\biggr) \\
&\qquad\qquad\underset{\varepsilon}{\approx} \tau(h) + 
\limsup_{N\to\infty}\biggl(\frac{1}{N^2}\log Z_{R,N}^h+\frac{n}{2}\log N\biggr)
\end{align*}
holds for all sufficiently large $m \in \mathbb{N}$ and all sufficiently small $\delta > 0$. Taking the limit as $m\to\infty$ and $\delta\searrow0$ we obtain $\chi(\tau) = \tau(h) + \pi_R(h)$, since $\varepsilon>0$ is arbitrary. The first assertion is immediate thanks to the general fact $\chi(\tau) \leq \eta_R(\tau)$, see \cite[Theorem 4.5 (1)]{Hiai:CMP05}. The second assertion is immediate due to Corollary \ref{C4.5}. 
\end{proof} 

\begin{remarks}\label{R4.6} (1) The above assumption ($\clubsuit$) is satisfied when the `\emph{empirical tracial state}' $f \in \mathcal{C}_R(\mathbf{x}) \mapsto \mathrm{tr}_N(f(\mathbf{A}))$ converges to $\tau$ in the weak* topology as $N\to\infty$, almost surely when $\mathbf{A}\in(M_N(\mathbb{C})_R^{sa})^n$
is distributed under $\lambda_{R,N}^h$. In fact, this implies a much stronger fact that 
$$
\lim_{N\to\infty}\lambda_{R,N}^h(\Gamma_R(\tau;N,m,\delta)) = 1
$$ 
for every $m \in \mathbb{N}$ and $\delta > 0$. 

(2) Assume that $\tau$ is the limit of the mean tracial states $\tau_{R,N}^h$ in the weak$^*$ topology as $N\to\infty$. If the $\tau$ were confirmed to be extremal (this is the case if $\chi(\tau)>-\infty$ due to the the degenerate convexity of $\chi$ \cite[$\chi.8$]{Voiculescu:Survey}), then Biane and Dabrowski's concentration lemma \cite[Lemma 6.1]{BianeDabrowski:AdvMath13} would imply the desired assumption ($\clubsuit$) as follows. The consequence of their lemma implies that for each $m\in \mathbb{N}$ and $\delta > 0$ one has $\lambda_{R,N}^h(\Gamma_R(\tau;N,m,\delta)) > 1/2$ for all sufficiently large $N \in \mathbb{N}$. This immediately implies the assumption ($\clubsuit$), so we have $\chi_\mathrm{orb}(\tau) = \tilde{\chi}_\mathrm{orb}(\tau)=\eta_{\mathrm{orb},R}(\tau)$ by \cite[Theorem 7.3\,(6)]{BianeDabrowski:AdvMath13} and the above proposition.

(3) It would be quite nice if one could prove the same consequence as in (2) above for any extremal weak$^*$-limit point of the mean tracial states $\tau_{R,N}^h$. This is suggested by the fact in quantum spin systems that a weak$^*$-limit point of local Gibbs states is a global Gibbs (equilibrium) states \cite[\S6.2.2]{BratteliRobinson:Book2}. The above proof of Proposition \ref{P4.7} does not work under the assumption that $\tau$ is a weak$^*$-limit point of $\tau_{R,N}^h$.     

(4) When a given tracial state $\tau$ is weak$^*$-exposed in the convex set of all $\sigma \in TS(\mathcal{C}_R(\mathbf{x}))$ conditioned by $\sigma\!\upharpoonright_{\mathcal{C}_R(\mathbf{x}_i)} = \tau_i$, $1 \leq i \leq n$, the orbital microstate counterpart of the consequence of Biane and Dabrowski's concentration lemma (which implies the assumption ($\spadesuit$)) holds with the essentially same proof as \cite[Corollary 5.4, Lemma 6.1]{BianeDabrowski:AdvMath13}. Hence it is desirable to find a suitable condition for a given tracial state to be weak$^*$-exposed in the above convex set rather than its extremality.      
\end{remarks}

\begin{example}\label{Ex4.8}\rm
According to \cite[Theorem 3.5]{GuionnetSegala:ALEA06}, if
$h=\sum_{i=1}^nx_i^2/2+\sum_{k=1}^l t_k q_k$ is a self-adjoint polynomial in
$\mathbf{x}$ with monomials $q_k$ and sufficiently small coefficients $t_k$, then there
exists a unique $\tau\in TS(\mathcal{C}_R(\mathbf{x}))$, given as a unique solution to the
Schwinger-Dyson equation, such that the `empirical tracial state'
$f\in\mathcal{C}_R(\mathbf{x})^{sa}\mapsto\mathrm{tr}_N(f(\mathbf{A}))$
converges to $\tau$ in the weak* topology as $N\to\infty$, almost surely when $\mathbf{A}\in(M_N(\mathbb{C})_R^{sa})^n$
is distributed under $\lambda_{R,N}^h$. This implies the stronger assumption in Remarks \ref{R4.6}\,(1) and that $\tau$ is the limit of $\tau_{R,N}^h$ in the weak$^*$ topology. Therefore, the tracial state $\tau$ enjoys the assertions of Proposition \ref{P4.7} and $\chi_\mathrm{orb}(\tau) =\tilde{\chi}_\mathrm{orb}(\tau)$ (thanks to $\chi(\tau) > -\infty$). Here it should be mentioned that the proof of Proposition \ref{P4.7} is essentially same as  that of \cite[Theorem 4.1]{GuionnetSegala:ALEA06}, where the formula $\chi(\tau) = \tau(h) + \pi_R(h)$ explicitly appears.
\end{example}

\section{Free independence and Orbital $\eta$-entropy}

In the notations as in \S3 we will establish the next theorem. 

\begin{theorem}\label{T5.1} Self-adjoint random multi-variables $\mathbf{X}_1,\dots,\mathbf{X}_n$ in a tracial $W^*$-probability space are freely independent and each of them has f.d.a.~if and only if $\eta_\mathrm{orb}(\mathbf{X}_1,\dots,\mathbf{X}_n) = 0$. 
\end{theorem}
 
We need the next transportation cost inequality for orbital equilibrium tracial states.  

\begin{proposition}\label{P5.2} Let $\tau \in TS(\mathcal{C}_R(\mathbf{x}))$ and set $\tau_i := \tau\!\upharpoonright_{\mathcal{C}_R(\mathbf{x}_i)}$, $1 \leq i \leq n$. If $\tau$ is an orbital equilibrium tracial state associated with some $h \in \mathcal{C}_R(\mathbf{x})^{sa}$, then $W_2(\tau,\bigstar_{i=1}^n \tau_i) \leq 4R\sqrt{r}\sqrt{-\eta_{\mathrm{orb},R}(\tau)}$ holds with $r := \max_{1\leq i \leq n} r(i)$, where $W_2$ denotes the free $2$-Wasserstein distance introduced by Biane and Voiculescu \cite{BianeVoiculescu:GAFA01}. 
\end{proposition}

Note that the constant $\sqrt{r}$ contained in the transportation cost inequality was missing in the proof of \cite[Proposition 4.4 (8)]{HiaiMiyamotoUeda:IJM09} (see \cite[Appendix]{Ueda:IUMJ1x} for an explanation on this minor error).

\begin{proof} The proof is just an adaptation of the method of \cite[Theorem 3.1]{HiaiUeda:IDQP06} into the present framework. Hence we give only its sketch. We will replace $\mathrm{U}(N)$ with $\mathrm{SU}(N)$ in the discussion below. The orbital free pressure $\pi_{\mathrm{orb},R}(h:(\tau_i)_{i=1}^n)$ and all the others appearing below do not change by this replacement.  

Firstly, we deal with the case when $\tau$ is a \emph{unique} (under $\tau\!\upharpoonright_{\mathcal{C}_R(\mathbf{x}_i)} = \tau_i$, $1 \leq i \leq n$) orbital equilibrium tracial state associated with $h$. One can choose a subsequence $N_k$ and a sequence $\mathbf{A}^{(k)} = (\mathbf{A}_i^{(k)})_{i=1}^n$ so that 
$$
\pi_{\mathrm{orb},R}(h,(\tau_i)_{i=1}^n) = 
\lim_{k\to\infty}\frac{1}{N_k}\log\int_{\mathrm{SU}(N_k)^n} d\gamma_{\mathrm{SU}(N_k)}^{\otimes n}(V_i)\,\exp\bigl(-N_k^2\mathrm{tr}_{N_k}(h((V_i\mathbf{A}_i^{(k)} V_i^*)_{i=1}^n))\bigr). 
$$
Similarly to the definitions in \S\S4.2, for every $f \in \mathcal{C}_R(\mathbf{x})^{sa}$ we define the Gibbs micro-ensembles $\mu_{N_k}^{(f,\mathbf{A}^{(k)})}$ on $\mathrm{SU}(N_k)^n$ with the normalizing constant $Z_{N_k}^{(f,\mathbf{A}^{(k)})}$, and also define the mean tracial state $\tau_{N_k}^{(h,\mathbf{A}^{(k)})} \in TS(\mathcal{C}_R(\mathbf{x}))$ as the mean of $f \mapsto \mathrm{tr}_{N_k}(f((V_i\mathbf{A}^{(k)}V_i^*)_{i=1}^n))$ with respect to $\mu_{N_k}^{(h,\mathbf{A}^{(k)})}$. By taking a further subsequence if necessary we may and do assume that $\tau_{N_k}^{(h,\mathbf{A}^{(k)})}$ converges to some $\tau_0 \in TS(\mathcal{C}_R(\mathbf{x}))$ in the weak* topology as $k\to\infty$. Consider $f \mapsto \log Z_{N_k}^{(f,\mathbf{A}^{(k)})}$ as a function on $\mathcal{C}_R(\mathbf{x})^{sa}$. It is a convex function (see the proof of Proposition 2.2\,(4)), and its G\^{a}teaux derivative at $h$ in the direction $f-h$ becomes $-N_k^2 \tau_{N_k}^{(h,\mathbf{A}^{(k)})}(f-h)$; hence
$$ 
\log Z_{N_k}^{(f,\mathbf{A}^{(k)})} = 
\log Z_{N_k}^{(h+(f-h),\mathbf{A}^{(k)})} 
\geq  
\log Z_{N_k}^{(h,\mathbf{A}^{(k)})} - N_k^2 \tau_{N_k}^{(h,\mathbf{A}^{(k)})}(f-h)
$$
for every $f \in \mathcal{C}_R(\mathbf{x})^{sa}$ (see \cite[p.297]{BratteliRobinson:Book2}). Therefore, dividing by $N_k^2$ and letting $k\to\infty$ we get $\tau_0(h) + \pi_{\mathrm{orb},R}(h:(\tau_i)_{i=1}^n) \leq \tau_0(f) + \pi_{\mathrm{orb},R}(f:(\tau_i)_{i=1}^n)$ for every $f \in \mathcal{C}_R(\mathbf{x})^{sa}$. This together with the uniqueness of $\tau$ implies that $\tau_{N_k}^{(h,\mathbf{A}^{(k)})}$ converges to $\tau$ itself in the weak* topology as $k \to \infty$. Then, as in \cite[Proposition 3.5]{HiaiMiyamotoUeda:IJM09} and \cite[Appendix]{Ueda:IUMJ1x} we have 
$$
W_2(\tau_{N_k}^{(h,\mathbf{A}^{(k)})},\tau_{N_k}^{(0,\mathbf{A}^{(k)})})^2 \leq 
16R^2 r \biggl(-\tau_{N_k}^{(h,\mathbf{A}^{(k)})}(h) 
- \frac{1}{N_k^2}\log Z_{N_k}^{(h,\mathbf{A}^{(k)})}\biggr), 
$$
and letting $k\to\infty$ we get the desired inequality, since the mean tracial state $\tau_{N_k}^{(0,\mathbf{A}^{(k)})}$ converges to the free product tracial state $\bigstar_{i=1}^n\tau_i$ in the weak* topology as $k\to\infty$ (see \cite[Lemma 3.3]{HiaiMiyamotoUeda:IJM09}).  

The not necessarily unique (under $\tau\!\upharpoonright_{\mathcal{C}_R(\mathbf{x}_i)} = \tau_i$, $1 \leq i \leq n$) orbital equilibrium case can be reduced to the previous one with the help of a standard method based on e.g., \cite[Lemma 6.2.43]{BratteliRobinson:Book2}, see the final part of the proof of \cite[Theorem 3.1]{HiaiUeda:IDQP06}. 
\end{proof} 

\begin{proof} (Theorem \ref{T5.1}) The `only if' part is Theorem \ref{T3.5}\,(8). If $\eta_\mathrm{orb}(\mathbf{X}_1,\dots,\mathbf{X}_n) = 0$, then $\tau_\mathbf{X}^{(R)}$ with $R := \max_{i,j}\Vert X_{ij}\Vert_\infty$ must be orbital equilibrium associated with $h=0$ so that the desired assertion immediately follows by the transportation cost inequality established in Proposition \ref{P5.2}.  
\end{proof}   

\section{A representation of $\chi_\mathrm{orb}$ as Legendre transform}

As shown in \cite[\S6]{Hiai:CMP05} the microstate free entropy $\chi$ can be written as the minus Legendre transform of the `double' free pressure. The importance of this representation of $\chi$ was explained in \cite[p.246--249]{Guionnet:LNM09}. Hence it is worthwhile to provide its $\chi_\mathrm{orb}$-counterpart for the reference in future research. 

Let $\mathbf{x}_i = (x_{ij})_{j=1}^{r(i)}$, $1 \leq i \leq n$, be non-commutative multi-indeterminates. For $R>0$ let $\mathcal{C}_R(\mathbf{x})$ be the universal $C^*$-free product introduced in \S2 and consider the minimal $C^*$-tensor product
$\mathcal{C}_R(\mathbf{x})\otimes_{\min}\mathcal{C}_R(\mathbf{x})$, whose norm is denoted
by the same $\Vert-\Vert_R$. When $\mathbf{A}=(\mathbf{A}_i)_{i=1}^n$,
$\mathbf{A}_i=(A_{ij})_{j=1}^{r(i)}\in(M_N(\mathbf{C})_R^{sa})^{r(i)}$, $1\le i\le n$, \
are given, one can define the $*$-homomorphism
$$
h\in\mathcal{C}_R(\mathbf{x})\otimes_{\min}\mathcal{C}_R(\mathbf{x})\mapsto
h(\mathbf{A})\in M_N(\mathbb{C})\otimes M_N(\mathbb{C})
$$
to be the tensor product of the $*$-homomorphism
$f\in\mathcal{C}_R(\mathbf{x})\mapsto f(\mathbf{A})\in M_N(\mathbb{C})$ (defined in \S2)
so that $(f\otimes g)(\mathbf{A})=f(\mathbf{A})\otimes g(\mathbf{A})$ for
$f,g\in\mathcal{C}_R(\mathbf{x})$.

\begin{definition}\label{D6.1} Let $\tau_i\in TS(\mathcal{C}(\mathbf{x}_i))$,
$1\le i\le n$, be given. For each $h=h^* \in \mathcal{C}_R(\mathbf{x})\otimes_\mathrm{min}\mathcal{C}_R(\mathbf{x})$ we define 
\begin{align*} 
&\pi_{\mathrm{orb},R}^{(2)}(h:(\tau_i)_{i=1}^n\,;N,m,\delta) \\
&\quad:= \sup_{\mathbf{A}_i \in \Gamma_R(\tau_i\,;N,m,\delta) \atop 1 \leq i \leq n} 
\log \int_{\mathrm{U}(N)^n} d\gamma_{\mathrm{U}(N)}^{\otimes n}(V_i)\,
\exp\bigl(-N^2(\mathrm{tr}_N\otimes\mathrm{tr}_N)(h((V_i\mathbf{A}_i V_i^*)_{i=1}^n)\bigr),
\\
&\pi_{\mathrm{orb},R}^{(2)}(h:(\tau_i)_{i=1}^n\,;m,\delta) 
:= 
\limsup_{N\to\infty} \frac{1}{N^2}\,\pi_{\mathrm{orb},R}^{(2)}(h:(\tau_i)_{i=1}^n\,;N,m,\delta),
\\
&\pi_{\mathrm{orb},R}^{(2)}(h:(\tau_i)_{i=1}^n) 
:= 
\lim_{m\to\infty\atop\delta\searrow0} \pi_{\mathrm{orb},R}^{(2)}(h:(\tau_i)\,;m,\delta)
= \inf_{m \in \mathbb{N}\atop\delta>0} \pi_{\mathrm{orb},R}^{(2)}(h:(\tau_i)\,;m,\delta), 
\end{align*}
where the first quantity should be read $-\infty$ when $\Gamma_R(\tau_i\,;N,m,\delta) = \emptyset$ for some $1 \leq i \leq n$. 
\end{definition}  

Let $\mathbf{X}_i=(X_{ij})_{j=1}^{r(i)}$, $1\le i\le n$, be self-adjoint random multi-variables
with $\|X_{ij}\|_\infty\le R$ in a $W^*$-probability space $(\mathcal{M},\tau)$. The
tracial states $\tau_\mathbf{X}^{(R)}\in TS(\mathcal{C}_R(\mathbf{x}))$ and
$\tau_{\mathbf{X}_i}^{(R)}\in TS(\mathcal{C}_R(\mathbf{x}_i))$, $1\le i\le n$, are defined
as in Theorem \ref{T3.4}. We then have the next representation of $\chi_\mathrm{orb}$.

\begin{proposition}\label{P6.1} With the assumption and the notations above,
\begin{align*}
&\chi_\mathrm{orb}(\mathbf{X}_1,\dots,\mathbf{X}_n) \\
&\quad= 
\inf\big\{(\tau_\mathbf{X}^{(R)}\otimes\tau_\mathbf{X}^{(R)})(h) + \pi_{\mathrm{orb},R}^{(2)}(h:(\tau_{\mathbf{X}_i}^{(R)})_{i=1}^n)\,\big|\, h=h^* \in \mathcal{C}_R(\mathbf{x})\otimes_\mathrm{min}\mathcal{C}_R(\mathbf{x}) \big\} \\
&\quad= 
\inf\big\{(\tau_\mathbf{X}^{(R)}\otimes\tau_\mathbf{X}^{(R)})(p) + \pi_{\mathrm{orb},R}^{(2)}(p:(\tau_{\mathbf{X}_i}^{(R)})_{i=1}^n)\,\big|\, p=p^* \in \mathbb{C}\langle\mathbf{x}\rangle\otimes\mathbb{C}\langle\mathbf{x}\rangle \big\} 
\end{align*}
without any assumption imposed on the $\mathbf{X}_i$. In particular, the infimum
expressions above are independent of the choice of $R\geq\max_{i,j}\|X_{ij}\|_\infty$. 
\end{proposition}
\begin{proof}
The proof below is essentially same as that of \cite[Theorem 6.4]{Hiai:CMP05}.
In the following we write $\tau_i=\tau_{\mathbf{X}_i}^{(R)}$ for simplicity and
$\mathbf{A}_i\in(M_N(\mathbb{C})_R^{sa})^{r(i)}$, $1\le i\le n$, for matricial
multi-microstates.
We may and do assume that all $\mathbf{X}_i$ have f.d.a.~by the definition of $\pi^{(2)}_{\mathrm{orb},R}$ and \cite[Theorem 2.6\,(2)]{Ueda:IUMJ1x}. Let $h=h^* \in \mathcal{C}_R(\mathbf{x})\otimes_\mathrm{min}\mathcal{C}_R(\mathbf{x})$ be arbitrarily given. For a given $\varepsilon>0$ one can choose an element 
$p=p^*\in\mathbb{C}\langle\mathbf{x}\rangle\otimes\mathbb{C}\langle\mathbf{x}\rangle$
(a non-commutative polynomial in `double' $\mathbf{x}$)
such that $\Vert h-p\Vert_R < \varepsilon$. Hence 
$$
\big|\pi_{\mathrm{orb},R}^{(2)}(h:(\tau_i)_{i=1}^n\,N,m,\delta) - \pi_{\mathrm{orb},R}^{(2)}(p:(\tau_i)_{i=1}^n\,N,m,\delta)\big| 
< N^2\varepsilon
$$
for every $N \in \mathbb{N}$, $m \in \mathbb{N}$ and $\delta>0$. (This is confirmed in the exactly same way as the proof of Proposition \ref{P2.2}\,(2).) Looking at $p$ one can choose $m_0 \in \mathbb{N}$ and $\delta_0 > 0$ in such a way that for every $N \in \mathbb{N}$  
$$
\big|(\mathrm{tr}_N\otimes\mathrm{tr}_N)(p((V_i\mathbf{A}_i V_i^*)_{i=1}^n)) - (\tau_\mathbf{X}^{(R)}\otimes\tau_\mathbf{X}^{(R)})(p)\big| < \varepsilon
$$
holds whenever $(V_i)_{i=1}^n \in \Gamma_\mathrm{orb}(\mathbf{X}_1,\dots,\mathbf{X}_n:(\mathbf{A}_i)_{i=1}^n\,;N,m,\delta)$ with $m \geq m_0$ and $0 < \delta \leq \delta_0$. Consequently,     
\begin{align*}
&\pi_{\mathrm{orb},R}^{(2)}(h:(\tau_i)_{i=1}^n\,;N,m,\delta) + N^2\varepsilon %\\&\quad
> \pi_{\mathrm{orb},R}^{(2)}(p:(\tau_i)_{i=1}^n\,;N,m,\delta) \\
&\quad\geq
\sup_{\mathbf{A}_i \in \Gamma_R(\tau_i\,;N,m,\delta) \atop 1 \leq i \leq n} 
\log \Big[\int_{\Gamma_\mathrm{orb}(\mathbf{X}_1,\dots,\mathbf{X}_n:(\mathbf{A}_i)_{i=1}^n\,;N,m,\delta)} d\gamma_{\mathrm{U}(N)}^{\otimes n}(V_i) \\
&\qquad\qquad\qquad\qquad\qquad\qquad\times
\exp\bigl(-N^2(\mathrm{tr}_N\otimes\mathrm{tr}_N)(p((V_i\mathbf{A}_i V_i^*)_{i=1}^n))\bigr)\Big] \\
&\quad\geq 
-N^2\big((\tau_\mathbf{X}^{(R)}\otimes\tau_\mathbf{X}^{(R)})(p) + \varepsilon\big) + 
\bar{\chi}_{\mathrm{orb},R}(\mathbf{X}_1,\dots,\mathbf{X}_n\,;N,m,\delta). 
\end{align*}
(See \cite[Equation (2.3)]{Ueda:IUMJ1x} for the definition of $\bar{\chi}_{\mathrm{orb},R}(\mathbf{X}_1,\dots,\mathbf{X}_n\,;N,m,\delta)$.) This implies by \cite[Proposition 2.4 and Corollary 2.7]{Ueda:IUMJ1x} that
\begin{align*}
\chi_\mathrm{orb}(\mathbf{X}_1,\dots,\mathbf{X}_n)
&\le(\tau_\mathbf{X}^{(R)}\otimes\tau_\mathbf{X}^{(R)})(p)
+ \pi_{\mathrm{orb},R}^{(2)}(h:(\tau_i)_{i=1}^n\,;m,\delta) + 2\varepsilon \\
&\le(\tau_\mathbf{X}^{(R)}\otimes\tau_\mathbf{X}^{(R)})(h)
+ \pi_{\mathrm{orb},R}^{(2)}(h:(\tau_i)_{i=1}^n\,;m,\delta) + 3\varepsilon.
\end{align*}
Letting $m\to\infty$ and $\delta\searrow0$ yields that $\chi_\mathrm{orb}(\mathbf{X}_1,\dots,\mathbf{X}_n)$ is not greater than the first infimum in the asserted identities.

Next, choose an arbitrary $\alpha > \chi_\mathrm{orb}(\mathbf{X}_1,\dots,\mathbf{X}_n) = \chi_{\mathrm{orb},R}(\mathbf{X}_1,\dots,\mathbf{X}_n)$. Then there exist $m \in \mathbb{N}$ and $\delta>0$ so that $\limsup_{N\to\infty}\frac{1}{N^2}\,\bar{\chi}_{\mathrm{orb},R}(\mathbf{X}_1,\dots,\mathbf{X}_n\,;N,m,\delta) < \alpha$. For any $\beta>0$ define
$p_{m,\beta}= p_{m,\beta}^*\in \mathbb{C}\langle\mathbf{x}\rangle\otimes\mathbb{C}\langle\mathbf{x}\rangle$ to be 
$$
\frac{\beta}{\delta^2}\sum_{l=1}^m \sum_{(i_k,j_k) \atop 1 \leq k \leq l} (x_{i_1 j_1}\cdots x_{i_l j_l} - \tau_\mathbf{X}^{(R)}(x_{i_1 j_1}\cdots x_{i_l j_l}))\otimes (x_{i_1 j_1}\cdots x_{i_l j_l} - \tau_\mathbf{X}^{(R)}(x_{i_1 j_1}\cdots x_{i_l j_l}))^*. 
$$
In what follows, we write $\Gamma((\mathbf{A}_i)_i,N,m,\delta) := \Gamma_\mathrm{orb}(\mathbf{X}_1,\dots,\mathbf{X}_n:(\mathbf{A}_i)_{i=1}^n\,;N,m,\delta)$ for short. Since
$(V_i)_{i=1}^n \not\in \Gamma((\mathbf{A}_i)_i\,;N,m,\delta)$ forces 
$$
(\mathrm{tr}_N\otimes\mathrm{tr}_N)(p_{m,\beta}((V_i\mathbf{A}_i V_i^*)_{i=1}^n) \geq \beta,
$$ 
we have
\begin{align*}
&\int_{\mathrm{U}(N)^n} d\gamma_{\mathrm{U}(N)}^{\otimes n}(V_i) \exp\bigl(-N^2(\mathrm{tr}_N\otimes\mathrm{tr}_N)(p_{m,\beta}((V_i\mathbf{A}_i V_i^*)_{i=1}^n)\bigr) \\
&\quad= 
\left(\int_{\Gamma((\mathbf{A}_i)_i,N,m,\delta)} + \int_{\Gamma((\mathbf{A}_i)_i,N,m,\delta)^c}\right)d\gamma_{\mathrm{U}(N)}^{\otimes n}(V_i) \\
&\qquad\qquad\qquad\qquad\times
\exp\bigl(-N^2(\mathrm{tr}_N\otimes\mathrm{tr}_N)(p_{m,\beta}((V_i\mathbf{A}_i V_i^*)_{i=1}^n)\bigr) \\
&\quad\leq \gamma_{\mathrm{U}(N)}^{\otimes n}(\Gamma((\mathbf{A}_i)_i,N,m,\delta)) + \exp(-N^2\beta).
\end{align*}
Since $(\tau_\mathbf{X}^{(R)}\otimes\tau_\mathbf{X}^{(R)})(p_{m,\beta}) = 0$, we therefore
obtain
\begin{align*}
&\inf\big\{(\tau_\mathbf{X}^{(R)}\otimes\tau_\mathbf{X}^{(R)})(p) + \pi_{\mathrm{orb},R}^{(2)}(p:(\tau_i)_{i=1}^n)\,\big|\, p=p^* \in \mathbb{C}(\mathbf{x})\otimes\mathbb{C}(\mathbf{x}) \big\} \\
%&\qquad\leq
%(\tau_\mathbf{X}^{(R)}\otimes\tau_\mathbf{X}^{(R)})(p_{m,\beta}) + 
%\frac{1}{N^2}\pi_{\mathrm{orb},R}^{(2)}(p_{m,\beta}:(\tau_{\mathbf{X}_i}^{(R)})) \\
&\quad\leq \limsup_{N\to\infty}\frac{1}{N^2}\pi_{\mathrm{orb},R}^{(2)}(p_{m,\beta}:(\tau_i)_{i=1}^n\,;N,m,\delta) \\
&\quad\leq 
\limsup_{N\to\infty}\log\bigl(\exp\bar{\chi}_{\mathrm{orb},R}(\mathbf{X}_1,\dots,\mathbf{X}_n\,;N,m,\delta) + \exp(-N^2\beta)\bigr)^{1/N^2} \\
&\quad\leq 
\limsup_{N\to\infty}\log\biggl[\exp\biggl(\frac{1}{N^2}\,\bar{\chi}_{\mathrm{orb},R}(\mathbf{X}_1,\dots,\mathbf{X}_n\,;N,m,\delta)\biggr) + \exp(-\beta)\biggr] \\
&\quad=
\log\biggl[\exp\biggl(\limsup_{N\to\infty}\frac{1}{N^2}\,\bar{\chi}_{\mathrm{orb},R}(\mathbf{X}_1,\dots,\mathbf{X}_n\,;N,m,\delta)\biggr) + \exp(-\beta)\biggr] \\
&\quad\leq 
\log\big(\exp\alpha + \exp(-\beta)\big) \searrow \alpha 
\end{align*}
as $\beta \nearrow \infty$. Thus the second infimum
is not greater than $\chi_\mathrm{orb}(\mathbf{X}_1,\dots,\mathbf{X}_n)$ in the desired identities.  
\end{proof}

\section{Orbital equilibrium tracial states  arising from Random matrix models}  
Let $\mathbf{u} = (u_i,u_i^*)_{i=1}^n$ (or $(u_i)_{i=1}^n$ for short) and $\mathbf{z}_i = (z_{ij})_{j=1}^{r(i)}$ be families of indeterminates and set $\mathbf{z} := \mathbf{z}_1\sqcup\cdots\sqcup\mathbf{z}_n$. Let $\mathbb{C}\langle\mathbf{u},\mathbf{z}\rangle$ be the universal $*$-algebra generated by those indeterminates subject to the relations $u_i u_i^* = u_i^* u_i = 1$, $1 \leq i \leq n$, and $z_{ij} = z_{ij}^*$, $1 \leq i \leq n$, $1 \leq j \leq r(i)$. For a given $R>0$ we also define $\mathcal{C}_R(\mathbf{u},\mathbf{z})$ to be the universal $C^*$-free product
$$
C(\mathbb{T})^{\star n} \star \big(\bigstar_{i=1}^n C[-R,R]^{\star r(i)}\big)
$$
with identifications $u_i(\zeta) = \zeta$ in the $i$th copy of $C(\mathbb{T})$ and $z_{ij}(t) = t$ in the $(i,j)$th copy of $C[-R,R]$. Note that $\mathbb{C}\langle\mathbf{u},\mathbf{z}\rangle$ canonically sits inside $\mathcal{C}_R(\mathbf{u},\mathbf{z})$ with keeping the same symbols of generators. There is a unique derivations $\partial_i : \mathbb{C}\langle\mathbf{u},\mathbf{z}\rangle \rightarrow \mathbb{C}\langle\mathbf{u},\mathbf{z}\rangle\otimes_\mathrm{alg}\mathbb{C}\langle\mathbf{u},\mathbf{z}\rangle$, $1\le i \le n$, determined by 
\begin{equation*}
\partial_i u_j := \delta_{ij} u_i\otimes1, \quad 
\partial_i u_j^* := -\delta_{ij} 1\otimes u_i^*, \quad 
\partial_i\!\upharpoonright_{\mathbb{C}\langle\mathbf{z}\rangle} := 0,
\qquad1\le j \le n,  
\end{equation*}
where $\mathbb{C}\langle\mathbf{z}\rangle \subset \mathbb{C}\langle\mathbf{u},\mathbf{z}\rangle$ is the unital $*$-subalgebra generated by $\mathbf{z}$. Let $\theta : \mathbb{C}\langle\mathbf{u},\mathbf{z}\rangle\otimes_\mathrm{alg}\mathbb{C}\langle\mathbf{u},\mathbf{z}\rangle \rightarrow \mathbb{C}\langle\mathbf{u},\mathbf{z}\rangle$ be defined by $\theta(a\otimes b):= ba$, and set $D_i :=\theta\circ\partial_i : \mathbb{C}\langle\mathbf{u},\mathbf{z}\rangle \rightarrow \mathbb{C}\langle\mathbf{u},\mathbf{z}\rangle$, $1\le i \le n$. 

Write $x_{ij} := u_i z_{ij} u_i^*$, $1 \leq i \leq n$, $1 \leq j \leq r(i)$, and set $\mathbf{x}_i := (x_{ij})_{j=1}^{r(i)}$, $1 \leq i \leq n$, and $\mathbf{x} := \mathbf{x}_1\sqcup\cdots\sqcup\mathbf{x}_n$. Fix a non-commutative polynomial $h=h^*$ in $\mathbf{x}$, which is understood as a self-adjoint element in $\mathbb{C}\langle\mathbf{u},\mathbf{z}\rangle$. Let a tracial state $\tau_0 \in TS(\mathcal{C}_R(\mathbf{z}))$ be given, and consider a tracial state $\tau_h \in TS(\mathcal{C}_R(\mathbf{u},\mathbf{z}))$ obtained as a solution to the so-called Schwinger-Dyson equation{\rm:} 
\begin{equation}\label{F8}
\tau_h\!\upharpoonright_{\mathcal{C}_R(\mathbf{Z})} = \tau_0, \quad 
(\tau_h\otimes\tau_h)\circ\partial_i(p) = \tau_h((D_i h) p), \quad 1 \leq i\leq n, 
\end{equation}
for every $p \in \mathbb{C}\langle\mathbf{u},\mathbf{z}\rangle$ {\rm(}$\subset \mathcal{C}_R(\mathbf{u},\mathbf{z})${\rm)}. The limit distributions of random matrix models studied in \cite{CollinsGuionnetSegala:AdvMath09} become such tracial states which also produce examples of orbital equilibrium tracial states in a suitable manner as follows.

\begin{example}\label{Ex7.1}\rm Assume that the above $h$ is of the form $h=\sum_{k=1}^l t_k q_k$ with monomials $q_k$ in $\mathbf{x}$ and sufficiently small coefficients $t_k$ and further that the above $\tau_0$ is given as the limit distribution of given deterministic matrices $\Xi(N) = \Xi_1(N)\sqcup\cdots\sqcup\Xi_n(N)$ as in \S\S4.2. Then \cite[Corollary 3.1]{CollinsGuionnetSegala:AdvMath09} shows that the `empirical tracial state' $f \in \mathcal{C}_R(\mathbf{u},\mathbf{z}) \mapsto \mathrm{tr}_N(f(\mathbf{V},\Xi(N)))$ converges to the unique tracial state $\tau_h$ determined by \eqref{F8} in the weak* topology as $N \to \infty$, almost surely when $\mathbf{V} = (V_i)_{i=1}^n$ is distributed under the Gibbs micro-emsemble $\mu_N^{(h,\Xi(N))}$ on $\mathrm{U}(N)^n$ defined in \S\S4.2. In particular, the empirical orbital tracial state $g \in \mathcal{C}_R(\mathbf{x}) \mapsto \mathrm{tr}_N(g((V_i\Xi_i(N)V_i^*)_{i=1}^n))$ converges to $\tau^h \in TS(\mathcal{C}_R(\mathbf{x}))$ induced from $\tau_h$ by the $*$-homomorphism $x_{ij} \in \mathcal{C}_R(\mathbf{x}) \mapsto u_i z_{ij} u_i^*\in \mathcal{C}_R(\mathbf{u},\mathbf{z})$ in the weak* topology as $N\to\infty$, almost surely when $(V_i)_{i=1}^n \in \mathrm{U}(N)^n$ is distributed under $\mu_N^{(h,\Xi(N))}$. Hence, by Proposition \ref{P4.2} with Remark \ref{R4.1} we see that $\tau^h$ is orbital equilibrium associated with $h$ under the hyperfiniteness assumption of every tracial state $\tau^h_i := \tau^h\!\upharpoonright_{\mathcal{C}_R(\mathbf{x}_i)}$  (this is the case when every $\Xi_i(N)$ is a singleton). This and \cite[Theorem 2.6]{HiaiMiyamotoUeda:IJM09} explain grounds for the final formula in the proof of \cite[Theorem 8.1]{CollinsGuionnetSegala:AdvMath09}. Moreover, it was shown in \cite[Theorem 7.1]{CollinsGuionnetSegala:AdvMath09} that $\lim_{N\to\infty}\frac{1}{N^2}\log Z_N^{(h,\Xi(N))}$ exists under the hypothesis that the $t_k$ are sufficiently small. This is nothing but 
\begin{align*}
%&
\pi_{\mathrm{orb},R}(h:(\tau^h_i)_{i=1}^n) %\\&\quad
= 
\lim_{N\to\infty}\frac{1}{N^2}\log
\int_{\mathrm{U}(N)^n} d\gamma_{\mathrm{U}(N)}^{\otimes n}(V_i)
\exp\bigl(-N^2 \mathrm{tr}_N(h((V_i \Xi_i(N) V_i^*)_{i=1}^n))\bigr),
\end{align*} 
and implies, by the proof of Proposition \ref{P4.2}, the following: 
$$
\chi_\mathrm{orb}(\tau^h) = \lim_{m\rightarrow\infty \atop \delta\searrow0}\liminf_{N\rightarrow\infty}\frac{1}{N^2}\log\gamma_{\mathrm{U}(N)}^{\otimes n}\big(\Gamma_\mathrm{orb}(\tau^h:\Xi(N)\,;N,m,\delta)\big).
$$ 
Furthermore, \cite[Theorem 7.1]{CollinsGuionnetSegala:AdvMath09} provides an expression of $\lim_{N\to\infty}\frac{1}{N^2}\log Z_N^{(h,\Xi(N))}$ combinatorially, and thus the orbital free entropy $\chi_\mathrm{orb}(\tau^h)$ admits a combinatorial expression.
\end{example}

Now, let $\tau_h\in TS(\mathcal{C}_R(\mathbf{u},\mathbf{z}))$ be as given in
\eqref{F8}. Via the GNS representation associated with $\tau_h$ we obtain a tracial $W^*$-probability space $(\mathcal{M},\tau)$, and the indeterminates $\mathbf{u}, \mathbf{z}$ give unitary random variables $\mathbf{U} = (U_i)_{i=1}^n$ and self-adjoint random multi-variables $\mathbf{Z} = \mathbf{Z}_1\sqcup\cdots\sqcup\mathbf{Z}_n$ with $\mathbf{Z} = (Z_{ij})_{j=1}^{r(i)}$, $1 \leq i \leq n$, in $(\mathcal{M},\tau)$. Write $X_{ij} := U_i Z_{ij} U_i^*$, $1 \leq i \leq n$, $1 \leq j \leq r(i)$, and set $\mathbf{X}_i := (X_{ij})_{j=1}^{r(i)}$, a self-adjoint multi-variable, for every $1 \leq i \leq n$ and $\mathbf{X} := \mathbf{X}_1\sqcup\cdots\sqcup\mathbf{X}_n$. Note here that $\tau^h$ in the above example is $\tau_\mathbf{X}^{(R)}$. We will compute the  liberation gradient $j(W^*(\mathbf{X}_i):W^*(\mathbf{X}_1,\dots,\hat{\mathbf{X}}_i,\dots,\mathbf{X}_n))$, $1 \leq i \leq n$, see \cite[\S5.4]{Voiculescu:AdvMath99}. In what follows, $\delta_i$ denotes the derivation of $W^*(\mathbf{X}_i)$ relative to $W^*(\mathbf{X}_1,\dots,\hat{\mathbf{X}}_i,\dots,\mathbf{X}_n)$ in the liberation theory (see \cite[\S5.3]{Voiculescu:AdvMath99}), and we define $\bar{\theta}(\sum_i a_i\otimes b_i) := \sum_i b_i a_i$ for every $\sum_i a_i\otimes b_i \in \mathcal{M}\otimes_\mathrm{alg}\mathcal{M}$.  

\begin{proposition}\label{P7.2} The liberation gradient $j(W^*(\mathbf{X}_i):W^*(\mathbf{X}_1,\dots,\hat{\mathbf{X}}_i,\dots,\mathbf{X}_n))$ becomes
$$
-U_i(D_i h)(\mathbf{X})U_i^* =  \bar{\theta}\circ\delta_i(h(\mathbf{X})) 
= \sum_{j=1}^{r(i)}\,[\bar{\theta}\circ\partial_{X_{ij}}(h(\mathbf{X})),X_i],
$$
where $\partial_{X_{ij}} = \partial_{X_{ij}:\mathbb{C}\langle X_{11},\dots,\hat{X}_{ij},\dots,X_{n r(n)}\rangle}$ is the free difference quotient associated with $X_{ij}$, $1 \leq j \leq r(i)$, $1 \leq i \leq n$, {\rm\cite[\S\S3.2]{Voiculescu:Survey}}.  
\end{proposition}

We should remark that the formula above is quite similar to \cite[Proposition 5.10, Corollary 8.3]{Voiculescu:AdvMath99}. This means that the work \cite{CollinsGuionnetSegala:AdvMath09} should have a deep connection to Voiculescu's liberation theory \cite{Voiculescu:AdvMath99}. The proof below is short enough, and hence we do give it for the reader's convenience.
 
\begin{proof} If $y \in \mathbb{C}\langle\mathbf{z}\rangle \subset \mathbb{C}\langle\mathbf{u},\mathbf{z}\rangle$, then $\partial_i (u_j y u_j^*) 
= \delta_{ij}(1\otimes u_i^*)(1\otimes u_i y u_i^* - u_i y u_i^* \otimes 1)(u_i\otimes1)$. Thus we have 
\begin{equation}\label{F9}
\delta_i\circ\pi\!\upharpoonright_{\mathbb{C}\langle\mathbf{x}\rangle} = - \pi((1\otimes u_i)\partial_i(\,-\,)(u_i^*\otimes1)),
\end{equation}
where we explicitly write the GNS representation $\pi : \mathcal{C}_R(\mathbf{u},\mathbf{z}) \to \mathcal{M}$ associated with $\tau_h$. For each $1\le i\le n$, the liberation gradient $j_i :=  j(W^*(\mathbf{X}_i):W^*(\mathbf{X}_1,\dots,\hat{\mathbf{X}}_i,\dots,\mathbf{X}_n))$ is determined as a unique element in $L^2(W^*(\mathbf{X}),\tau\!\upharpoonright_{W^*(\mathbf{X})})$ (which naturally sits inside the bigger $L^2$-space $L^2(\mathcal{M},\tau)$) with the equation:
\begin{equation}\label{F10}
\tau(j_i m) = (\tau\otimes\tau)\circ\delta_i(m), \qquad m \in \mathbb{C}\langle\mathbf{X}\rangle. 
\end{equation}
Hence, for each non-commutative polynomial $p$ in $\mathbf{x}$ we observe by \eqref{F9} that 
\begin{gather*}
(\tau\otimes\tau)\circ\delta_i(p(\mathbf{X})) 
= -(\tau_h\otimes\tau_h)((1\otimes u_i)\partial_i(p)(u_i^*\otimes 1)) = -(\tau_h\otimes\tau_h)((u_i^*\otimes 1)\partial_i(p)(1\otimes u_i)), \\
(u_i^*\otimes1)\partial_i(p)(1\otimes u_i) 
= 
\partial_i(u_i^* p u_i) + 1\otimes u_i^* p u_i - u_i^* p u_i\otimes1, 
\end{gather*}
and hence by the Schwinger-Dyson equation 
\begin{equation*}
(\tau_h\otimes\tau_h)((u_i^*\otimes 1)\partial_i(p)(1\otimes u_i)) = (\tau_h\otimes\tau_h)\circ\partial_i(u_i^* p u_i) 
=
\tau_h((u_i(D_i h)u_i^*)p). 
\end{equation*}
Therefore, it follows that identity \eqref{F10} is equivalent to $\tau(j_i p(\mathbf{X})) =-\tau_h((u_i(D_i h)u_i^*)p)$ for every non-commutative polynomial $p$ in $\mathbf{x}$, $1\leq i\leq n$. It is easy to confirm that the restriction of $\pi\circ\mathrm{Ad}\,u_i\circ D_i$ to the unital $*$-subalgebra generated by $\mathbf{x}$ is exactly $-\bar{\theta}\circ\delta_i\circ\pi$. Thus, we arrive at $j_i = -\pi(\mathrm{Ad}u_i(D_i h)) = \bar{\theta}\circ\delta_i(h(\mathbf{X})) \in W^*(\mathbf{X})$. The last equality in the desired formula immediately follows from Voiculescu's unpublished work (see \cite[Equation (1) in p.3665]{Dabrowski:JFA10}).
\end{proof} 

\section*{Acknowledgment} 

The second-named author visited at the Fields Institute in July 2013, where part of the present work was done. He would like to express his sincere thanks to the organizers of the free probability program for inviting him and to the Fields institute for providing comfortable atmosphere. One of the referees kindly adviced us to make the presentation more attractive, and the other gave many fruitful and helpful comments which enable us to improve, especially, Remarks \ref{R4.6}\,(2)--(4) and Proposition \ref{P7.2}. We would like to thank the referees for their contributions.

}

\end{document}